\documentclass{amsart}
\newcommand{\plusindent}{\null~~~~~~~~~~~~~}
\newcommand{\BOLD}[1]{{\bf #1}}

\def\ITEMIZEplusONE{} 
\def\ENUMplusONE{} 
\usepackage{amsmath}
\usepackage{amssymb}
\usepackage{mathrsfs}
\usepackage{enumerate}
\usepackage{color}

\newtheorem{theorem}{Theorem}[section]
\newtheorem{lemma}[theorem]{Lemma}
\newtheorem{proposition}[theorem]{Proposition}

\newtheorem{claim}[theorem]{Claim}

\theoremstyle{definition}
\newtheorem{definition}[theorem]{Definition}
\newtheorem{observation}[theorem]{Observation}

\theoremstyle{remark}
\newtheorem{remark}[theorem]{Remark}
\newtheorem{notation}[theorem]{Notation}
\newtheorem{question}[theorem]{Question}
\newtheorem{fact}[theorem]{Fact}

\newcommand{\rest}{{\restriction}}
\newcommand{\Dom}{{\rm Dom}}
\newcommand{\seb}{{\rm sb}}
\newcommand{\psb}{{\rm psb}}
\newcommand{\fil}{{\rm fil}}
\newcommand{\almfrt}{\mbox{\rm alm-frt}}
\newcommand{\suc}{{\rm suc}}
\newcommand{\ufbd}{{\rm ufbd}}
\newcommand{\vfbd}{{\rm vfbd}}
\newcommand{\CWT}{{\rm CWT}}
\newcommand{\CTW}{{\rm CWT}}

\newcommand{\CON}{{\rm CON}}
\newcommand{\ZFC}{{\rm ZFC}}
\newcommand{\Fin}{{\rm Fin}}
\newcommand{\RK}{{\rm RK}}
\newcommand{\frt}{{\rm frt}}
\newcommand{\af}{\mbox{\rm alm-frt}}
\newcommand{\rt}{{\rm rt}}

\newcommand{\Dp}{{\rm Dp}}

\newcommand{\ut}{{\rm ut}}
\newcommand{\bd}{{\rm bd}}
\newcommand{\Rang}{{\rm Rang}}

\newcommand{\wilog}{{\rm without loss of generality}}
\newcommand{\Wilog}{{\rm Without loss of generality}}

\newcommand{\then}{{\emph{then}}}
\newcommand{\Then}{{\emph{Then}}}

\newcommand{\Iff}{{\emph{iff}}}

\newcommand{\cA}{{\mathscr A}}

\newcommand{\cB}{{\mathscr B}}

\newcommand{\bbN}{{\mathbb N}}

\newcommand{\gu}{{\mathfrak u}}

\newcommand{\gd}{{\mathfrak d\/}}

\newcommand{\cH}{{\mathscr H}}

\newcommand{\cI}{{\mathscr I}}
\newcommand{\cJ}{{\mathscr J}}

\newcommand{\bbP}{{\mathbb P}}
\newcommand{\bbQ}{{\mathbb Q}}
\newcommand{\cP}{{\mathscr P}}

\newcommand{\cT}{{\mathscr T}}

\newcommand{\cU}{{\mathscr U}}

\newcommand{\cW}{{\mathscr W}}

\newcommand{\bfx}{\mathbf{x}}

\newcount\skewfactor
\def\mathunderaccent#1#2 {\let\theaccent#1\skewfactor#2
\mathpalette\putaccentunder}
\def\putaccentunder#1#2{\oalign{$#1#2$\crcr\hidewidth
\vbox to.2ex{\hbox{$#1\skew\skewfactor\theaccent{}$}\vss}\hidewidth}}
\def\name{\mathunderaccent\tilde-3 }

\newenvironment{PROOF}[2][\proofname.]
   {\begin{proof}[#1]\renewcommand{\qedsymbol}{\textsquare$_{\rm #2}$}}
{\end{proof}}

\numberwithin{equation}{section}

\usepackage[hidelinks]{hyperref}

\begin{document}
\makeatletter\def\shfiuwefootnote{\gdef\@thefnmark{}\@footnotetext}\makeatother\shfiuwefootnote{Version 2022-10-14\_2. See \url{https://shelah.logic.at/papers/980/} for possible updates.}

\title{Nice $\aleph_1$ generated non-$P$-points, I, 980}

\author{Saharon Shelah}

\address{Einstein Institute of Mathematics\\
Edmond J. Safra Campus, Givat Ram\\
The Hebrew University of Jerusalem\\
Jerusalem, 91904, Israel\\
and \\
Department of Mathematics\\
Hill Center - Busch Campus \\
Rutgers, The State University of New Jersey \\
110 Frelinghuysen Road \\
Piscataway, NJ 08854-8019 USA}

\email{shelah@math.huji.ac.il}

\urladdr{http://shelah.logic.at}




\thanks{Partially supported by National Science Foundation, Grant 0600940 and the US-Israel Binational Science Foundation, Grant 2010405 Paper 980 on the Author's list. The author thanks Alice Leonhardt for the beautiful typing of earlier versions (up to 2019) and in later versions the  author would like to thank the typist for his work and is also grateful for the generous funding of typing services donated by a person who wishes to remain anonymous. References like \cite[2.7=La32]{Sh:945} means the label of Th.2.7 is a32.  The reader should note that the version in my website is usually more updated than the one in the mathematical archive.}

\date{September 29, 2022}  

\begin{abstract}
    We define a family of non-principal ultrafilters on $\bbN$ which are, in a sense, very far from $P$-points. We prove the existence of such ultrafilters under reasonable conditions.  In subsequent articles we shall prove that such ultrafilters may exist while no $P$-point exists. Though our primary motivations came from forcing and independence results, the family of ultrafilters introduced here should be interesting from combinatorial point of view too.
    
    We aim in a subsequent paper to use this to show e.g. the consistency of ``$\mathfrak{u} = \aleph_{1}$ and no $P$-point'' (with $\ZFC$). We have wrote done the following easier statement (in 
    E104): 
    the ultra filters we constructed under CH for ``reasonable'' forcing like Sacks forcing preserve the ultra filter. 
\end{abstract}

\maketitle

\newpage

\section{Introduction}

One of the important notions in general topology and set theory of the reals is that of a $P$-point. Recall that {\em a $P$-point\/} is a non-principal ultrafilter $D$ on $\bbN$ with the property that for any countable family $\cA\subseteq D$ there is a $B\in D$ almost (modulo finite) included in all $A\in \cA$ (see Definition~\ref{z9}). Concerning these and other special
ultrafilters on $\bbN$, their history and basic applications we refer the reader to the survey article by Blass \cite{Sh:10}. 

In many applications it is important to preserve $P$-points by specific forcing notions and by a forcing iterated with countable supports. Recall
that {\em preservation of an ultrafilter\/} means that the ultrafilter from the ground model $\mathbf V$ generates an ultrafilter in the generic extension $\mathbf V[G]$ (see \cite[Chapter VI]{Sh:f}). We have a very good understanding of these questions and many relevant results have been presented in the literature. From our point of view the $P$-points are
tractable for independence results because of the following fact:

\begin{fact}[Nice properties of P-points]\label{boxplusone}

    \ENUMplusONE\begin{enumerate}[(A)]
        \item\label{bp1.A}  there are quite many forcing notions preserving $P$-points, 
        
        \item\label{bp1.B}  a proper forcing notion $\bbQ$ which preserves ``$D$ is an ultrafilter'' preserves its being a $P$-point,
        
        \item\label{bp1.C}  the preservation of $P$-points is preserved in limits of CS iterations.
        
        \item\label{bp1.D}  We can destroy a $P$-point by forcing, i.e., ensure it has no extension to a $P$-point (and consequently we may prove the consistency of ``there are no $P$-points''),
        
        \item\label{bp1.E}  moreover, we can ``split hairs'', i.e., destroy some $P$-points while preserving other, so we can have unique $P$-point up to isomorphisms. 
    \end{enumerate}
    
    (Already the properties (\ref{bp1.A},\ref{bp1.B},\ref{bp1.C}) give  a well controlled way to have ultrafilters generated by $\aleph_1 < 2^{\aleph_0}$ sets). 
\end{fact}

For more details we refer the reader to \cite[Ch.VI and
Ch.XVIII,\S4]{Sh:f}. 

We may wonder if the theory developed for $P$-points can be repeated for other ultrafilters. We may ask:

\begin{question}\label{z2}
    Are there other types of ultrafilters preserved by  CS iterations of suitable forcing notions? 
    
    In particular, we are interested in preservation of our ultrafilters at limit stages of CS iterations: for a limit ordinal $\delta$, having been preserved by $\bbP_\alpha$ for $\alpha<\delta$, does this hold for $\bbP_\delta$ when $\langle\bbP_\alpha,\name{\bbQ}_\beta: \alpha\le\delta, \beta<\delta\rangle$ is a CS iteration of proper forcing notions?
\end{question}

We suggested this problem in \cite[3.13]{Sh:666} and we speculated about it there. Note that ultrafilters as in Question~\ref{z2} for CS iterations are
naturally generated by $\aleph_1$ sets; moreover CS iterations are mainly interesting when we start with CH, and ``preserve an ultrafilter'' is meaningful only when we add reals, naturally $\aleph_2$ ones. We suspect
this direction is related to the question on the existence of a point of van Douwen cite{vD} (see Question~\ref{boxplus2a} below), but at present we do not know neither if they are related nor how to answer it. Other specific problems that we have in mind when developing the theory for Question~\ref{z2} are a
problem of Nyikos and a problem of Dow:

\begin{question}\label{boxplus2a}
    [E.~van Douwen] Is it consistent that: there is no ultrafilter $D$ on $\bbQ$ such that every $A \in D$ contains a member of $D$ which is a closed set with no isolated points?
\end{question}

\begin{question}\label{boxplus2b}
    [P.~Nyikos] Is it consistent to have some
    ultrafilter $D \in \beta^*(\bbN)\setminus\mathbb N$ of character $\aleph_1$, but no $P$-point?
\end{question}

\begin{question}\label{boxplus2c}
    [A.~Dow] Is it consistent to have ${\gu} =
    \aleph_1$, there is a $P$-point $D$, but no $P$-point $D$ with $\chi(D) = \aleph_1$?
\end{question}

In the series of papers started here the main points are:

\ENUMplusONE\begin{enumerate}
    \item[(A)] we have an involved family of sets (really well founded trees) appearing in the definition,
    
    \item[(B)] each ultrafilter has no $P$-point as a quotient,
    
    \item[(C)] they are related to a game,
    
    \item[(D)] such systems exists assuming, e.g., $\diamondsuit_{\aleph_1}$,
    
    \item[(E)] enough relevant forcing notions preserve such systems, in particular, some serving \ref{boxplusone}(\ref{bp1.C}), so answering the first question in \ref{z2},
    
    \item[(F)]  we have a preservation theorem for such
    ultrafilters under CS iterations,
    
    \item[(G)] As an application, we will solve Nyikos' problem~\ref{boxplus2b}.   
\end{enumerate}

So problems~\ref{z2} and~\ref{boxplus2b} will be resolved by the methods we start developing here, but presently not~\ref{boxplus2a} and~\ref{boxplus2c} (a problem of van Douwen and a problem of Dow).

In the present article we define ultrafilters analogous to $P$-points but with no $P$-point as a quotient; this is done in Sections 2 and 3. In the fourth section we deal with basic connections to forcing that we will use in the independence results in subsequent papers.

In the second paper of the series (still ``work in progress'')  we present these ultrafilters in a more general framework and deal with sufficient conditions for such an ultrafilter to generate an
ultrafilter in a suitable generic extension. For the limit case we continue the proof of preservation theorems in \cite{Sh:f}, in particular
\cite[Ch.VI,1.26,1.27]{Sh:f} and Case A with transitivity of \cite[Ch.XVIII,\S3]{Sh:f}.  For the successor case we need that the relevant
forcing preserves our ultrafilters.  We will conclude with the proof for $\CON(\gu=\aleph_1+$ no $P$-points). In 
  \cite{Sh:E104}  
we deal with the consistency of the preservation of an ultrafilter by e.g. Sacks forcing CS support of them and more. 

Noting that  the ultrafilters so far were really analogous to selective (i.e., Ramsey) ultrafilters we plan to give a more general framework which also includes $P$-points in a  planned  third part.

\begin{remark}\label{z7}
    There may be $P$-point while $\gd > \aleph_1$, see Blass and Shelah \cite{Sh:242} and references there, but the existence of ultrafilters in the direction here, far from $P$-point, implies ${\gd} = \aleph_1$, see the survey of Blass \cite{Bls10}. But note that the ultrafilter may be $\aleph_1$-generated in a different sense: union of $\aleph_1$ families of the form $\fil(B) \cap \cP(\max(B))$.
    
    Note that it may be harder (than in the $P$-point case) to build such ultrafilters as here which are $\mu$-generated instead of $\aleph_1$-generated because of the unbounded countable depth involved. We have not looked at this as well as at the natural variants of our definition (not to speak of generalization to reasonable ultrafilters, see \cite{Sh:830} and Ros{\l}anowski and Shelah \cite{Sh:889}, \cite{Sh:890}).
\end{remark}

\newpage

\section{System of filters using well founded trees}

\begin{notation}\label{a2}
    Here, $M=(M,<_M)$ is a partial order and $B$ is a subset of $M$ inheriting its order. 
    
    For $\eta\in B$ we let $B_{\ge \eta}=\{\nu \in B:\eta \le_M \nu\}$ and  similarly $B_{> \eta}$. We also define  \[\suc_B(\eta)=\{\nu\in B:\eta <_M \nu\mbox{ but for no $\rho\in B$ do  we have }\eta<_M\rho<_M\nu\}\] and $\max(B)=\{\nu\in B:B \cap M_{>\nu}=\emptyset\}$. 
    
    We say that $Y$ is {\em a front of $B\subseteq M$\/}
    iff:  $Y \subseteq B$ and every branch (maximal chain) of $B$ meets $Y$ and the members of $Y$ are pairwise $<_M$-incomparable. 
\end{notation}

\begin{definition}\label{a5-1}
    Let $M=(M,<_M)$ be a partial order. 
    A set $B \subseteq M$ is a \emph{countable well-founded sub-tree of~$M$}\/  if the following conditions (\ref{51a})--(\ref{51f}) are satisfied.
    
    \ENUMplusONE\begin{enumerate}[(a)]
        \item \label{51a} The set  $B$ is a countable subset of~$M$.
        
        \item \label{51b} The set $B$ has a $<_M$-minimal member called its
        root, $\rt(B)$.
        
        \item \label{51c} The structure $B$ (i.e., $(B,<_M \rest B)$) is a tree with $\le\omega$ levels and no $\omega$-branch (so all chains in $B$ are finite). 
        
        \item \label{51d} For each $\nu \in B$ the set $\suc_B(\nu)$ is either empty or infinite.
        
        \item \label{51e} If $\eta,\nu\in B$ are $<_M$--incomparable, then they have no common $\leq_M$--upper bound (i.e., they are incompatible not only in $B$ but even in $M$).  We abbreviate 
        this as  $\eta\parallel_M\nu$. 
        
        \item \label{51f} If $\nu\in B\setminus \max(B)$ and $F\subseteq M\setminus M_{\leq \nu}$ is finite, then for infinitely many $\varrho\in \suc_B(\nu)$ we have $(\forall \rho\in F)(\rho \parallel_M \varrho)$.
    \end{enumerate}
    
    The family of all countable well-founded  sub-trees of $M$ is denoted by $\CWT(M)$. 
\end{definition}

We will define a natural filter  on the set of maximal nodes of every countable well-founded tree $B$; this filter will naturally induce Rudin-Keisler images on each front of $B$. 

\begin{definition}\label{a5-2}
    For $B \in \CWT(M)$ let $\frt(B)$ be the set of all fronts of~$B$, which in this case means the family of all maximal sets of pairwise incomparable members of~$B$.
    
    For antichains $Y_1,Y_2$ of $M$ we say that $Y_2$ is \emph{above} $Y_1$  iff: \[(\forall \eta \in Y_2)(\exists\nu\in Y_1)[\nu \le_M\eta].\]  This will be used mainly for $Y_1,Y_2 \in \frt(B)$, $B \in \CWT(M)$. 
    
    For $Y_1,Y_2$ as above let the projection $h_{Y_1,Y_2}$ be the unique function $h :Y_2\longrightarrow Y_1$ such that $h(\eta)\le_M \eta$ for $\eta \in Y_2$. 
    
    If $Y_1,Y_2\in\frt(B)$ \then \, $Y_2$ {\em is almost above\/} $Y_1$ iff:
    
    \begin{quote}
        for some $B'\in\seb(B)$, see \ref{a5-3} below, $B' \cap Y_2$ is above $B' \cap Y_1$.
    \end{quote}
    
    We also define the projection $h_{Y_1,Y_2}$ as above, but its domain  is not $Y_2$ but the set $\{\eta \in Y_2:(\exists\nu\in  Y_1)(\nu \le_M \eta)\}$.
    
    The default value of $Y \in \frt(B)$ is max$(B)=
    \{\nu\in B:\nu$ is $<_M$-maximal in~$B\}$.
\end{definition}

We now define two notions of largeness for  subtrees.  \emph{Exhaustive} subtrees correspond to filter sets or ``measure 1'' sets, \emph{positive} subtrees
will correspond to the notion ``positive modulo a filter'' or ``not in the ideal dual to the filter''.

\begin{definition}\label{a5-3}
    Let $B\in\CWT(M)$. We call  $B'$ is an {\em exhaustive subtree\/} of $B$ iff:
    
    \ENUMplusONE\begin{enumerate}[(a)]
        \item\label{a53a}  $B'\in\CWT(M)$, $B'\subseteq B$,
        
        \item\label{a53b}  $\rt(B')=\rt(B)$,
        
        \item\label{a53c}  for all $\nu \in B'$ we have:   $\suc_{B'}(\nu)\subseteq \suc_B(\nu)$ and $\suc_B(\nu)\backslash \suc_{B'}(\nu)$ is finite.
    \end{enumerate}
    
    We let $\seb(B)$ be the set of all exhaustive subtrees $B'$ of~$B$, and we 
    say $f$ \emph{witnesses} ``$B'\in \seb(B)$'' if $f:B' \backslash\max(B)\longrightarrow [B]^{< \aleph_0}$ satisfies \[\nu\in B'\backslash\max(B)\quad\Rightarrow\quad \suc_B(\nu)\backslash \suc_{B'}(\nu)\subseteq f(\nu).\] 
    
    Note that for $f$ being a witness only $f \rest B'$ matters; in fact only  the restriction  $f
    \rest\{\nu\in B'\mid \exists\eta\in Y: \nu\le \eta\}$ matters when we are interested in $D_{B,Y}$.
    
    For $B\in\CWT(M)$ and $Y\in\frt(B)$ let $E_{B,Y}$ be the filter on $Y$ generated by the family  \[\{Y \cap B':B'\mbox{ is an exhaustive subtree of $B$, i.e., }B'\in \seb(B)\}.\] 
    
    
    For $B\in\CWT(M)$ let $\psb_M(B)$ (``p'' stands
    for positive) be the set of \emph{positive subtrees} $B'$ of $B$ which means  (\ref{a53a}),(\ref{a53b}) as above and
    
    \ENUMplusONE\begin{enumerate}[(c)']
        \item \label{a53cprime}  if $\nu\in B'\backslash\max(B)$, then $\suc_{B'}(\nu)$ is an infinite subset of $\suc_B(\nu)$. 
    \end{enumerate}
\end{definition}

\begin{definition}\label{a5-4}
    An antichain $Y\subseteq M$ is an {\em almost front of $B$\/} if for some $B' \in \seb(B)$ the intersection $Y \cap B'$ is a front of~$B'$. Let $\af(B)=\af_M(B)$ denote the set of all almost fronts of~$B$. 
    
    For $Y \in \almfrt_M(B)$ let \[\fil_M(Y,B) = \{X \subseteq Y: \mbox{ for some $B'\in \seb(B)$ we have }X \supseteq B' \cap Y\}.\] 
\end{definition}

\begin{definition}\label{a5-56}
    Let $\le^*_M$ be the following two-place relation (actually a partial order) on $CWT(M)$: \\
    
    \relax $B_1 \le^*_M B_2$ \Iff \  
    \begin{minipage}[t]{88mm}
    \relax $(B_1,B_2 \in \CWT(M)$, \\  $\rt(B_1) = \rt(B_2)$, \\ 
    and for some $B'_2 \in \seb(B_2)$, we have 
    
    \begin{itemize}
        \item[--]  $B'_2 \cap B_1 \in \psb_M(B_1)$, and
        
        \item[--]  every almost front of $B'_2 \cap B_1$ is an almost front of~$B_2$. 
    \end{itemize}
    \end{minipage}
    \\
    
    The tree $B'_2$ as above will be called {\em a witness for $B_1\leq^*_M B_2$}. 
    
    For $B\in\CWT(M)$, {\em the depth of $B$\/} is defined recursively by \[\Dp(B)=\sup\{\Dp(B_{\ge\eta}) + 1:\eta\in B \backslash\{\rt(B)\}\}.\]   
    
\end{definition}

\begin{remark}\label{a7}
    If $B,B'\in\CWT(M)$, $B'\subseteq B$ and $\nu\in B'$, then $\suc_B(\nu) \cap
    B'\subseteq\suc_{B'}(\nu)$, but the two sets do not have to be equal. Note
    that in the definitions of both $B'\in\seb(B)$ and $B'\in \psb_M(B)$ we do
    require that 
    \[
     \big(\forall\nu\in B'\big)\big(\suc_B(\nu)\cap B'=
    \suc_{B'}(\nu)\big) \]
    This condition implies that if $Y\subseteq B$ is a front of $B$, then $Y\cap B'$ is a front of~$B'$.
\end{remark}

\begin{observation}\label{a11}
    Let $M$ be a partial order and $B,B_1,B_2\in\CWT(M)$.
    
    \ENUMplusONE\begin{enumerate} 
        \item\label{28.1} We have that $B_1 \le^*_M B_2$ if and only if every almost front of $B_1$ is an almost front of~$B_2$. 
        
        \item\label{28.2} The relation $\leq^*_M$ is a partial order on $\CWT(M)$.
        
        \item\label{28.3} If $B_2\in\psb_M(B_1)$, \then\  $B_1\leq^*_M B_2$ and $\psb_M
        (B_2)\subseteq \psb_M(B_1)$.
        
        \item\label{28.4} If $B_2\in \seb(B_1)$, \then\  $B_2\in\psb(B_1)$, $\seb(B_2)\subseteq \seb(B_1)$ and $B_1\leq^*_M B_2\leq ^*_M B_1$. 
          
        \item\label{28.5} For $B \in \CWT(M)$, {\rm max}$(B)$ is a front of $B$ and also $\{\rt(B)\}$ is. If $B\neq\{\rt(B)\}$, \then\, $\suc_B(\rt(B))$ is a front
        of~$B$.  
        
        \item\label{28.6} Every front of $B \in \CWT(M)$ is an almost front of~$B$.
        
        \item\label{28.7} If $B \in \CWT(M)$ \then\, $\Dp(B)$ is a countable ordinal and $B_{\geq \eta}\in \CWT(M)$ for all $\eta\in B$.
        
        \item\label{28.8} If $Y\subseteq B\setminus\{\rt(B)\}$ is a front of $B$, and $\eta\in\suc_B(\rt(B))$, \then\, $Y\cap B_{\geq\eta}$ is a front of~$B_{\geq\eta}$. 
        
        \item\label{28.9} If $Y$ is an almost front of $B$ and an antichain $Z$ is an almost front of $B_{\geq\eta}$ for every $\eta\in Y\cap B$, \then\, $Z$ is an almost front of~$B$. 
        
        \item\label{28.10} If $B_1\leq^*_M B_2$ and $Y$ is a front of $B_1$, \then\, there is $B_2'\in\seb(B_2)$ such that $Y\cap B_2'$ is a front of $B_2'$ and $(B_1)_{\geq\eta}\leq^*_M (B_2')_{\geq\eta}$ for all $\eta\in Y\cap B_2'$.
    \end{enumerate}
\end{observation}

\begin{PROOF}{\ref{a11}}
    Straightforward.
\end{PROOF}

\begin{definition}\label{7g.1}
    Let $\mathbf K$ be the class of the objects $\mathbf x=\langle M_{\mathbf x},
    <_{M_{\mathbf x}}, \bar{\cA}_{\mathbf x},\cA_{\mathbf x}, \cB_{\mathbf x},
    \leq_{\mathbf x} \rangle$ satisfying the following properties (\ref{71.a})--(\ref{71.h}). 
    \ENUMplusONE\begin{enumerate}[(a)]
    \item\label{71.a} The structure  $ (M_{\mathbf x},<_{M_{\mathbf x}})=(M ,<)$ is a partial order with the smallest element $\rt_{\mathbf x} = \rt(\mathbf x)$. Let $M^-_{\mathbf x}=M_{\mathbf x}\backslash \{\rt_{\mathbf x}\}$, 
    \item\label{71.b}  $\bar{\cA}_{\mathbf x}=\bar{\cA}= \langle{\cA}_\eta:\eta
    \in M \rangle = \langle {\cA}^{\mathbf x}_\eta:\eta \in M_{\mathbf x}\rangle$
    and $\cA_{\mathbf x}=\bigcup\{\cA_\eta:\eta\in M^-_{\mathbf x}\}$,
    \item\label{71.c}   ${\cA}_\eta \subseteq \CWT(M)$, let ${\cA}^-_\eta =
    {\cA}_\eta \backslash \{\{\eta\}\}$,
    \item\label{71.d}   $\rt(B)=\eta$ for every $B\in\cA_\eta$,
    \item\label{71.e}   ${\cA}_\eta$ is not empty, in fact $\{\eta\}\in {\cA}_\eta$,
    \item\label{71.f}   $\cB_{\mathbf x}=\cA^{\mathbf x}_{\rt_{\mathbf x}}\setminus\big\{
    \{\rt_{\mathbf x}\}\big\}$ and $\le_{\mathbf x}$ is a directed partial order
    on ${\cB}_{\mathbf x}$,
    \item\label{71.g}   $B_1 \le_{\mathbf x} B_2$ implies $B_1\le^*_M B_2$, see
    Definition~\ref{a5-56} and, of course, $B_1,B_2 \in \cB_{\mathbf x}$,
    \item\label{71.h}   if $\nu\in B\in {\cA}_\eta$ then $B\cap M_{\ge\nu}\in
    {\cA}_\nu$.
    \end{enumerate}
    When dealing with $M_{\mathbf x}, \bar{\cA}_{\mathbf x}$ etc we may omit $\mathbf x$ when clear from the context.
\end{definition}

\begin{definition}\label{7g.4}
    Let $\mathbf x \in \mathbf K$ and $\eta \in M_{\mathbf x}$.
    \ENUMplusONE\begin{enumerate}
    \item Let $\frt(\eta) = \frt_{\mathbf x}(\eta) = \{Y:Y$ is a front of $B$ for some $B \in \cA^{\mathbf x}_\eta\}$ and instead of $\frt(B)$ (see Definition~\ref{a5-2}) we may write also $\frt_{\mathbf x}(B)$.  We let 
    \[\frt^-(\eta) = \{Y \in \frt(\eta):Y \ne \{\eta\}\}.\] Omitting $\eta$ means  $\eta = \rt_{\mathbf x}$. 
    
    \item Similarly, using Definition~\ref{a5-4}, we define $\af_{\mathbf x}(\eta)$
    (and $\af_{\mathbf x}$).
    
    \item Let $B\in \cA^{\mathbf x}_\eta$. We define
    \[\begin{array}{ll}
    \Fin(B)=\big\{f: &f\mbox{ is a function with domain $B \backslash
    \max(B)$ such that}\\
    &f(\nu) \in [\suc_B(\nu)]^{< \aleph_0}\mbox{ for all }\nu\in B\setminus
    \max(B)\big\},
    \end{array}\]
    and for $f\in \Fin(B)$ we set 
    \[A_f = A_{B,f} = \big\{\eta \in B: \big(\forall\rho\in B\setminus\max(B)\big)
    \big(\forall\varrho\in \suc_B(\rho)\big)
    \big(\varrho\leq_M\eta\ \Rightarrow\ \varrho \notin f(\rho)\big)
    \big\}.\]
    (Recall Definition~\ref{a5-3}.)
    \item Assume that $Y\in\af_{\mathbf x}$. We let $D_Y = D^{\mathbf x}_Y$ be the
      family 
    \begin{align*}
    \big\{\,Z \subseteq Y:\  &\mbox{for some }B\in\cB_{\mathbf x}\mbox { and } B'\in
    \seb(B)\\ & \mbox{ we have } Y\in\af(B)\mbox{ and }B'\cap Y\subseteq Z\,\big\}.
    \end{align*}
    \item If $B \in \cB_{\mathbf x}$, then $D_{\mathbf x}(B)=D^{\mathbf x}_{\max(B)}$.
    \item We let $\Dp_{\mathbf x}(\eta) = \sup\{\Dp(B)+1: B \in \cA^{\mathbf
        x}_\eta\}$ (recall Definition~\ref{a5-56}). 
    \end{enumerate}
    If ${\mathbf x}$ is clear from the context, then we may omit the
    subscript/superscript ${\mathbf x}$ in the objects defined above. 
\end{definition}

Let us recall the definition of the Rudin--Keisler order on ultrafilters.

\begin{definition}\label{z9rk}
    Let $D_\ell$ be an ultrafilter on $\cU_\ell$ for $\ell = 1,2$. We say $D_1
    \le_{\RK} D_2$ \Iff \, there is a function $h$ whose domain and range 
    are subsets of $\cU_2$, $\cU_1$,  respectively, such that
    \[ \forall A \subseteq \cU_1 : \quad   A \in D_1\   \Leftrightarrow\ 
    \{a\in \Dom(h):h(a) \in A\} \in D_2 \]
\end{definition}

\begin{observation}\label{7g.5g}
    Assume $\mathbf x \in \mathbf K$ and  let $B,B_1,B_2 \in\cB_{\mathbf x}$. 
    \ENUMplusONE\begin{enumerate} 
    \item\label{75.1} The singleton $\{\rt_{\mathbf x}\}$ is in $ \frt_{\mathbf x}$ and $D^{\mathbf x}_{\{\rt_{\mathbf
          x}\}} = \big\{\{\rt_{\mathbf x}\}\big\}$.  
    \item\label{75.2} If $B_1\leq_{\mathbf x} B_2$, $f\in\Fin(B_1)$ and $Y\in\af(B_1)$,
    \then\, $Y\in\af(B_2)$ and there is $g\in\Fin(B_2)$ such that $Y\cap A_{B_2,
    g}\subseteq Y\cap A_{B_1,f}$.
    \item\label{75.3} If $Y\in\af(B_\ell)$, $f_\ell\in \Fin(B_\ell)$ (for $\ell=1,2$),
      \then\, there are $B^*\in\cB_{\mathbf x}$ and $g\in\Fin(B^*)$ such that $B_1\leq_{\mathbf x} B^*$, $B_2\leq_{\mathbf x} B^*$ and
    \[Y\cap A_{B^*,g}\subseteq Y\cap A_{B_1,f_1}\cap A_{B_2,f_2}.\]
    \item\label{75.4} If $Y\in\af_{\mathbf x}$, \then\, $D_Y^{\mathbf x}$ is a filter on~$Y$.
    \item\label{75.5} If $B_1 \le_{\mathbf x} B_2$, $Y_1\in\af(B_1)$, and $Y_2=
    Y_1\cap B_2$ (hence $Y_2\in\af(B_2)$), \then\, $Y_2\in D^{\mathbf x}_{Y_1}$
    and $D^{\mathbf x}_{Y_2} = D^{\mathbf x}_{Y_1}\rest Y_2$.
    \item\label{75.6} Assume that $Y_1,Y_2 \in\frt(B)$ and $Y_2$ is above~$Y_1$. Let $h:Y_2
    \stackrel{\rm onto}{\longrightarrow} Y_1$ be the projection, i.e.,
    \[h(\nu_2)=\nu_1\quad \Leftrightarrow\quad \nu_1\in Y_1\ \wedge\
    \nu_2 \in Y_2\ \wedge\ \nu_1 \le_{M_{\mathbf x}} \nu_2.\]
    \Then\, $h(D_{Y_2})=D_{Y_1}$, i.e., $D_{Y_1} = \big\{A\subseteq Y_1:
    h^{-1}[A]\in D_{Y_2} \big\}$ (so $h$ witnesses $D_{Y_1}\le_{\RK}
    D_{Y_2}$). 
    \item\label{75.7}  If $B_1 \le_{\mathbf x} B_2$ and $Y_\ell=\suc_{B_\ell}(\rt_{\mathbf x})$
      for $\ell=1,2$, \then \,: 
    \ENUMplusONE\begin{enumerate}
    \item[(a)] $Y_\ell$ is a front of $B_\ell$ and $Y_1$ almost above $Y_2$, see Definition~\ref{a5-2},
    \item[(b)] if $Y$ is a front of $B_\ell$ and it is not $\{\rt_{\mathbf x}\}$,
    then $Y$ is above $Y_\ell$.
    \end{enumerate}
    \item\label{75.8} The set  $\max(B)$ is the maximal front of $B$ which means that it
    is above any other.
    \item\label{75.9} If $\bbQ$ is an ${}^\omega\omega$-bounding forcing and $B
    \in\cB_{\mathbf x}$, \then \, for any $B'\in \seb(B)^{\mathbf V[\bbQ]}$ there is $B'' \in (\seb(B))^{\mathbf V}$ such that $B'' \subseteq B'$.
    \item\label{75.10} If $F$ is a finite subset of $M_{\mathbf x}^-$, $B\in \cB_{\mathbf x}$,
    \then\, there is a branch (i.e., a maximal chain) $C\subseteq B$ such that
    \[(\forall\rho\in F)(\forall \sigma\in C)(\rho\nleq_M \sigma).\]
    \item\label{75.11}  If $B\in\cA_\eta$ and $\nu\in B\setminus\max(B)$, \then\, ${\rm
    id}_{\mathbf x}(\nu,B)$ is a proper ideal ideal on $\suc_B(\nu)$.
    \end{enumerate}
\end{observation}

\begin{PROOF}{\ref{7g.5g}}
    Straightforward.
\end{PROOF} 

\begin{definition}\label{7g.7} 
    \ENUMplusONE\begin{enumerate}
    \item\label{77.1} For an (infinite) cardinal $\kappa$ let $\mathbf K_{< \kappa}$ be
    the class of $\mathbf x \in \mathbf K$ such that $\|\mathbf x\| := |M_{\mathbf
    x}| + \sum\{|{\cA}^{\mathbf x}_\eta|:\eta \in M_{\mathbf x}\} < \kappa$,
    similarly $\mathbf K_{\le \kappa}$.
    \item\label{77.2} The relation $\le_{\mathbf K}$ is the following two-place relation on $\mathbf K$ (it is a partial order, see Observation~\ref{7g.9} below): $\mathbf x \le_{\mathbf K}
    \mathbf y$ if and only if
    \ENUMplusONE\begin{enumerate}[(a)]
    \item\label{7g7.a} $M_{\mathbf x} \subseteq M_{\mathbf y}$ (as partial orders) and,
    moreover, for any $\eta,\nu\in M_{\mathbf x}$ we have
    \[\nu\parallel_{M_{\mathbf x}} \eta\quad\mbox{ if and only if }\quad
    \nu\parallel_{M_{\mathbf y}}\eta,\]
    and
    
    \item\label{7g7.b} $\eta \in M_{\mathbf x}\quad \Rightarrow\quad
    {\cA}^{\mathbf x}_\eta\subseteq {\cA}^{\mathbf y}_\eta$, and
    
    \item\label{7g7.c} $\rt_{\mathbf y}=\rt_{\mathbf x}$ (actually follows from (\ref{7g7.d})), and
    
    \item\label{7g7.d} $\le_{\mathbf x}=\le_{\mathbf y} \rest \cB_{\mathbf x}$.
    \end{enumerate}
    
    \item\label{77.3} If $\langle \mathbf x_\alpha:\alpha < \delta\rangle$ is a $\le_{\mathbf K}$-increasing sequence we define $\mathbf x_\delta =
    \bigcup\{\mathbf x_\alpha:\alpha < \delta\}$, the union of the sequence,
    by $M_{\mathbf x_\delta}=\bigcup\{M_{\mathbf x_\alpha}:
    \alpha < \delta\}$ as partial orders and ${\cA}^{\mathbf x_\delta}_\eta=\bigcup\{{\cA}^{\mathbf x_\alpha}_\eta:
    \alpha<\delta$ satisfies $\eta \in M_{\mathbf x_\alpha}\}$ and $\le_{\mathbf x_\delta}=\bigcup\{\le_{\mathbf x_\alpha}:\alpha<\delta\}$.
    \end{enumerate}
\end{definition}

\begin{observation}\label{7g.9}
    \ENUMplusONE\begin{enumerate}
        \item\label{79.1}
        It is easy to see that the relation $\le_{\mathbf K}$ is really a partial order.  
        \item\label{79.2}
        Moreover, this order is
        closed under chains, i.e.: \\
        Whenever $\langle \mathbf x_\alpha:\alpha < \delta\rangle$ is $\le_{\mathbf K}$--increasing, we can define $\mathbf x_\delta$ as 
        the union of the sequence.  It is then clear that $\mathbf x_\delta  $  is 
        a $\le_{\mathbf K}$--lub of the sequence and $\|\mathbf x_\delta\|
        \le \sum\|\{\|\mathbf x_\alpha\|:\alpha < \delta\}$.
    \end{enumerate}
\end{observation}

\begin{definition}\label{7g.14}
    Let $\mathbf x \in \mathbf K$. We say that $\mathbf x$ is:
    \ENUMplusONE\begin{enumerate}
    \item[{\bf fat}] iff: if $B\in\cB_{\mathbf x}$ and $B'\in
    \seb(B)$, then there is $B''\in\seb(B')$ such that $B''\in
    \cB_{\mathbf x}$ and $B \le_{\mathbf x} B''$;
    \item[{\bf big}] iff: if $B\in\cB_{\mathbf x}$ and $\mathbf c:\max(B)
    \longrightarrow \{0,1\}$, then for some $B'\in\cB_{\mathbf x}$ we have
    that
    \[B'\in\psb_{M_{\mathbf x}}(B)\cap\cB_{\mathbf x},\quad B \le_{\mathbf x} B', \quad
    \mbox{ and }\quad {\mathbf c}\rest \max(B')\mbox{ is constant},\]
    \item[{\bf large}] iff: whenever $B\in\cB_{\mathbf x}$ and $\mathbf c$ is a
       function with domain $\max(B)$, then for some $B'\in\psb_{M_{\mathbf
           x}}(B)\cap\cB_{\mathbf x}$ and a front $Y$ of $B'$ we have $B\leq_{\mathbf x} B'$ and 
    \[\big(\forall\eta,\nu\in\max(B')\big)\big({\mathbf c}(\eta) = {\mathbf c}
    (\nu)\ \Leftrightarrow\ (\exists \rho \in Y)(\rho \le_{M_{\mathbf x}}
    \eta \wedge \rho \le_{M_{\mathbf x}} \nu)\big),\]
    \item[{\bf full}] iff: whenever $B\in\cA^{\mathbf x}_\eta$, $\eta
    \ne\rt_{\mathbf x}$ and $B'\in \psb_{M_{\mathbf x}}(B)$, then $B'\in
    \cA^{\mathbf x}_\eta$.
    \end{enumerate}
\end{definition}

\newpage

\section{Construction of ultra-systems}

\begin{lemma}\label{8h.3}
    The set $\mathbf K_{\le \aleph_0}$ is non-empty.
\end{lemma}

\begin{PROOF}{\ref{8h.3}}
    Define ${\mathbf x}$ so that $M_{\mathbf x}=\{\eta_*\}$, ${\cA}^{\mathbf x}_{\eta_*}=\{\{\eta_*\}\}$, ${\rm rt}_{\mathbf x}=\eta_*$.
    Now it is easy to check.
\end{PROOF}

\begin{lemma}\label{8h.7}
    If $\mathbf x \in \mathbf K$ and $\eta \in M_{\mathbf x}$ satisfies $|{\cA}^{\mathbf x}_\eta| = 1$, i.e., ${\cA}^{\mathbf x}_\eta=\big\{\{\eta\}\big\}$, \then \, for some $\mathbf y \in\mathbf K$ we have $\mathbf x \le_{\mathbf K} \mathbf y$, $|{\cA}^{\mathbf y}_\eta|>1$ and $\|\mathbf y\|\le\|\mathbf x\|+\aleph_0$.
\end{lemma}

\begin{PROOF}{\ref{8h.7}}
    Let $\langle \eta_n:n < \omega\rangle$ be pairwise distinct objects
    not belonging to $M_{\mathbf x}$. We define $\mathbf y$ by:
    \ENUMplusONE\begin{enumerate}
    \item[(a)] $M_{\mathbf y}$ has set of elements $M_{\mathbf x} \cup
    \{\eta_n:n < \omega\}$,
    \item[(b)] $\nu<_{M_{\mathbf y}}\rho$ if and only if $\nu
    <_{M_{\mathbf x}}\rho$ or $\nu \le_{M_{\mathbf x}}\eta\ \wedge\
    (\exists n)(\rho = \eta_n)$,
    \item[(c)] ${\cA}^{\mathbf y}_\nu$ is defined by a case distinction:
    \begin{itemize}
    \item[--] If $\nu\in M_{\mathbf x} \backslash \{\eta\}$, then ${\cA}^{\mathbf y}_\nu:=\cA^{\mathbf x}_\nu$.
    \item[--] If $\nu = \eta$, then $ {\cA}^{\mathbf y}_\nu:=\{\{\eta\},\{\eta_n:n<\omega\}\cup \{\eta\}\}$.
    \item[--] If $\nu = \eta_n$, then ${\cA}^{\mathbf y}_\nu:=\{\{\eta_n\}\}$.
    \end{itemize}
    \item[(d)] the order $\leq_{\mathbf y}$ is $\leq_{\mathbf x}$ if $\eta\ne\rt_{\mathbf x}$, and it is determined by:\\ $\{\eta\}\leq_{\mathbf y}
      \{\eta_n:n<\omega\}\cup\{\eta\}$ if $\eta = \rt_{\mathbf x}$.
    \end{enumerate}
    Now check.
\end{PROOF}

\begin{lemma}\label{8h.10}  
    \ENUMplusONE\begin{enumerate}
    \item\label{810.1} If $\mathbf x \in \mathbf K_{\le \aleph_0}$ \then \, for some $\mathbf y \in
      \mathbf K_{\le \aleph_0}$ we have $\mathbf x \le_{\mathbf K} \mathbf y$ and in $\cB_{\mathbf y}$ there is a $\le_{\mathbf y}$--maximal member.
    \item\label{810.2} If $\mathbf x\in\mathbf K_{\le \aleph_0}$ and some $B\in \cB_{\mathbf x}$ is $\le_{\mathbf x}$--maximal \then \, for some $\mathbf y\in \mathbf K_{\le
        \aleph_0}$ and $B' \in\cB_{\mathbf y}$ we have $\mathbf x\le_{\mathbf K}\mathbf
      y$ and  $B <_{\mathbf y} B'$.
    \item\label{810.3}  If $\mathbf x\in \mathbf K_{\leq\aleph_0}$, $\eta\in M_{\mathbf x}$, $B_1\in\cA^{\mathbf x}_\eta$, $B_2 \in \psb_{M_{\mathbf x}}(B_1)$ and
    \[\eta=\rt_{\mathbf x}\ \Rightarrow\ B_1 \mbox{ is $\leq_{\mathbf
        x}$--maximal},\] 
    \then \, there is $\mathbf y \in \mathbf K_{\le \aleph_0}$ such that ${\mathbf
      x}\le_{\mathbf K} {\mathbf y}$ and $B_2 \in \cA^{\mathbf y}_\eta$.
    \item\label{810.4} If $\mathbf x\in \mathbf K_{\leq\aleph_0}$, $B_1\in\cB_{\mathbf x}$ and $B_2\in\seb(B_1)$, \then \, there is $\mathbf y\in \mathbf K_{\le \aleph_0}$
      such that ${\mathbf x}\le_{\mathbf K} {\mathbf y}$ and $B_2 \in \cB_{\mathbf y}$.
    \end{enumerate}
\end{lemma}

\begin{PROOF}{\ref{810.1}}
    If in $({\cB}_{\mathbf x},\le_{\mathbf x})$ there is a maximal member then we let $\mathbf y = \mathbf x$. Otherwise, as it is directed (see clause (\ref{71.f}) of Definition~\ref{7g.1}) and $\|\mathbf x\| \le\aleph_0$ (because $\mathbf x \in \mathbf K_{\le \aleph_0}$), there is a strictly $\leq_{\mathbf x}$--increasing cofinal sequence $\langle B_n:n<\omega\rangle$.  Let $Y_n = \suc_{B_n}(\rt_{\mathbf x})$.

    Note that for each $m_1<m_2$, the set $Y_{m_1} \cap B_{m_2}$ is an almost front of $B_{m_2}$ (so also it is almost above $Y_{m_2}$). Hence for $m_1<m_2\le n$ we have that $Y_{m_1} \cap B_n$ is an almost front of $B_n$ which is almost above $Y_{m_2} \cap B_n$. Consequently we may choose $B^*_n\in \seb(B_n)$ such that each $Y_\ell\cap B^*_n$ is a front of $B^*_n$ and $Y_\ell\cap B^*_n$ is above $Y_{\ell+1}\cap B^*_n$ (for all $\ell<n$). Moreover, we may also require that
    
    \begin{align}\label{otimeszero}
    \text{ for each $\ell<n$ and $\eta\in Y_\ell\cap B^*_n$ we have $(B_\ell)_{\geq\eta}\leq^*_{M_{\mathbf x}}(B^*_n)_{\geq\eta}$}
    \end{align}
    (remember Observation~\ref{a11}(\ref{28.10})).
    
    Fix a list $\langle\rho_\ell:\ell<\omega\rangle$ of all members of $M_{\mathbf
      x}$ (possibly with repetitions). By induction on $n<\omega$ choose $\nu_n$
    such that
    \begin{align}
    \label{x.a} &\nu_n\in Y_n\cap B^*_n=\suc_{B^*_n}(\rt_{\mathbf x})\\
     \label{x.b}&\text{if $\ell<n$, then $\nu_n,\nu_\ell$ are $<_{M_{\mathbf
            x}}$--incompatible (i.e., $\nu_\ell\parallel_{ M_{\mathbf x}}\nu_n$),}\\
     \label{x.c}&\text{if $\ell<n$ and $\rho_\ell\neq\rt_{\mathbf x}$, then $\rho_\ell\parallel_{M_{\mathbf x}} \nu_n$}.
    \end{align}
    [Why is the choice possible? By the demand (\ref{51f}) of Definition~\ref{a5-1}
    applied to $\nu=\rt_{\mathbf x}$ and $F=\{\nu_\ell,\rho_\ell:\ell<n\}
    \setminus \{\rt_{\mathbf x}\}$.] 
    
    We define
    \[B^*=\{\rt_{\mathbf x}\}\cup\bigcup\{B^*_n\cap (M_{\mathbf x})_{\geq \nu_n}:
    n<\omega\}.\] 
    This set $B^*$ is clearly a countable well-founded tree, $B^*\in\CWT(M_{\mathbf x})$  with root $\rt_{\mathbf x}$ and $\suc_{B^*}(\rt_{\mathbf x})= \{\nu_n:n<\omega\}$.
    
    [Why? It should be clear that conditions (\ref{51a})--(\ref{51d}) of
    Definition~\ref{a5-1} hold, $\rt(B^*)=\rt_{\mathbf x}$ and $\suc_{B^*}(\rt_{\mathbf x})= \{\nu_n:n<\omega\}$. To verify clause (\ref{51e}) 
    suppose $\eta,\nu\in B^*$ are $<_{M_{\mathbf x}}$--incomparable. Then both $\eta\neq\rt_{\mathbf x}$ and $\nu\neq\rt_{\mathbf x}$, so $\eta,\nu\in
    \bigcup\limits_{n<\omega} (B^*)_{\nu_n}$. If, for some $n$, we have $\eta,\nu\in B^*_n\cap (M_{\mathbf x})_{\geq\nu_n}$, then they are $<_{M_{\mathbf x}}$--incompatible as $B^*_n\subseteq B_n$ and $B_n$ satisfies
    \ref{a5-1}(\ref{51e}). Otherwise, for some distinct $\ell,n$ we have $\eta\in
    B^*_\ell\cap (M_{\mathbf x})_{\geq\nu_\ell}$ and $\nu\in B^*_n\cap (M_{\mathbf
      x})_{\geq\nu_n}$.
      Now, if we could find $\rho\in M_{\mathbf x}$
      such that $\rho\geq_{M_{\mathbf x}} \eta$
      and $\rho\geq_{M_{\mathbf x}}\nu$,
      then $\nu_\ell,\nu_n$ would be compatible contradicting \eqref{x.b},
      so $B^*$ indeed satisfies clause (\ref{51e}) of Definition~\ref{a5-1}.
      Finally, to verify (\ref{51f}) suppose $\nu\in B^*\setminus\max(B^*)$
      and $F\subseteq M_{\mathbf x} \setminus (M_{\mathbf x})_{\leq\nu}$ is finite.
      If $\nu_n \leq_{M_{\mathbf x}} \nu$ for some $n$, then the properties of $B^*_n$ apply.
      So suppose $\nu=\rt_{\mathbf x}$. Choose $m$ so that $F\subseteq \{\rho_\ell:\ell<m\}$
      and use condition \eqref{x.c} to argue that for all $n\geq m$ and $\rho\in F$ we have $\nu_n\parallel_{M_{\mathbf x}}\rho$.]
    
    Also:
    \begin{quote}
     $B\leq^*_{M_{\mathbf x}} B^*$ for all $B\in\cB_{\mathbf x}$.
    \end{quote} 
    [Why? Since $\leq^*_{M_{\mathbf x}}$ is a partial order and by the choice of $B_n$, it is enough to show that for each $n<\omega$ we have $B_n\leq^*_{M_{\mathbf x}} B^*$, i.e., that every almost front of $B_n$ is an
    almost front of~$B^*$. To this end suppose that $Z\subseteq B_n$ is an
    almost front of $B_n$ for some $n<\omega$. If $Z=\{\rt_{\mathbf x}\}$, then
    there is nothing to do, so suppose $Z\subseteq B_n\setminus\{\rt_{\mathbf
      x}\}$, i.e., $Z\subseteq\bigcup\{(B_n)_{\geq\rho}:\rho\in Y_n\}$.
    Plainly, the set
    \[X=\{\rho\in Y_n: Z\mbox{ is not an almost front of }
    (B_n)_{\geq\rho}\}\]
    is finite and hence for some $m>n$ we have $X\subseteq
    \{\rho_\ell:\ell<m\}$. Then for every $k>m$ we have: 
    \ENUMplusONE\begin{enumerate}[(a)]
    \item\label{en.a} The element $\nu_k$ is incompatible with every $\nu\in X$,
    \item\label{en.b} The set $Y_n\cap (B^*_k)_{\geq \nu_k}$ is a front of $(B^*_k)_{\geq \nu_k}$,
    \item\label{en.c} $(B_n)_{\geq\eta}\leq^*_{M_{\mathbf x}} (B^*_k)_{\geq\eta}$ 
    for every $\eta\in Y_n\cap (B^*_k)_{\geq \nu_k}$ (by \eqref{otimeszero}),
    \item\label{en.d} The set $Z\cap (B_n)_{\geq\eta}$ is an almost front of $(B_n)_{\geq\eta}$ for every $\eta\in Y_n\cap
    (B^*_k)_{\geq \nu_k}$, and thus
    \item\label{en.e} The set  $Z\cap (B^*_k)_{\geq\eta}$ is an almost front of $(B_k^*)_{\geq\eta}$ for every $\eta\in Y_n\cap
    (B^*_k)_{\geq \nu_k}$.
    \item\label{en.f} Finally, $Z$ is an almost front of $(B^*_k)_{\geq\nu_k}$ (by Observation~\ref{a11}(\ref{28.9}) and (\ref{en.b})+(\ref{en.e})).
    \end{enumerate}
    Since $\suc_{B^*}(\rt_{\mathbf x})=\{\nu_k:k<\omega\}$, we know that $\{\nu_k:m<k<\omega\}$ is an almost front of~$B^*$. Therefore, by
    Observation~\ref{a11}(\ref{28.9}) and (\ref{en.f}), we conclude that $Z$ is an almost front of $B^*$.]

    Lastly, we define $\mathbf y$:
    \begin{itemize}
    \item[--] $(M_{\mathbf y},<_{M_{\mathbf y}})= (M_{\mathbf x},
    <_{M_{\mathbf x}})$,
    \item[--] ${\cA}^{\mathbf y}_\nu=\cA^{\mathbf x}_\nu$ iff: $\nu \in M_{\mathbf x} \backslash \{\rt_{\mathbf x}\}$, and ${\cA}^{\mathbf y}_{\rt_{\mathbf x}} = \cA^{\mathbf x}_{\rt_{\mathbf x}} \cup
    \{B^*\}$, 
    \item[--] $B_1 \le_{\mathbf y} B_2$ if and only if $B_1\le_{\mathbf x} B_2$ or $B_1 \in A^{\mathbf y}_{\rt_{\mathbf x}} \wedge B_2 = B^*$. 
    \end{itemize}
    It should be clear that ${\mathbf y}\in {\mathbf K}_{\leq\aleph_0}$ is as
    required.
\end{PROOF}

\begin{proof}[Proof of Lemma \ref{8h.10}(\ref{810.2}),(\ref{810.3}),(\ref{810.4})]
    Straightforward; see also Lemmas~\ref{8h.11},
    \ref{8h.15} below.
\end{proof}

\begin{lemma}\label{8h.11}
    Assume that ${\mathbf x}\in {\mathbf K}_{{\leq}\aleph_0}$ and $B\in \cB_{\mathbf
    x}$ is $\leq_{\mathbf x}$--maximal. Then for some ${\mathbf y}\in {\mathbf K}_{{\leq}\aleph_0}$ and $B'\in \cB_{\mathbf y}$ we have
    
    \ENUMplusONE\begin{enumerate}
        \item[(a)] ${\mathbf x}\leq {\mathbf y}$, $M_{\mathbf x}= M_{\mathbf y}=M$, and
        
        \item[(b)] $B'\in \cB_{\mathbf y}$ is $\leq_{\mathbf y}$--maximal,
        
        \item[(c)] if $\nu\in B'\setminus \max(B')$ and $\rho\in M\setminus  M_{\leq\nu}$, then for all but finitely many $\varrho\in\suc_{B'}(\nu)$ we have $\rho \parallel_M \varrho$ or for all 
        but 
        finitely many $\varrho \in \suc_{B'}(\nu)$ we have $\rho \lhd_{M} \varrho.$ 
    \end{enumerate}
\end{lemma}

\begin{PROOF}{\ref{8h.11}}
      Fix a list $\langle\rho_\ell:\ell<\omega\rangle$ of all members of $M_{\mathbf x}$ (possibly with repetitions). For each $\eta\in B\setminus
      \max(B)$ by induction on $n<\omega$ we choose $\nu_{\eta,n}$ such that
      
    \ENUMplusONE\begin{enumerate}
        \item[--] $\nu_{\eta,n}\in\suc_B(\eta)$,
        
        \item[--] $\nu_{\eta,n}\neq\nu_{\eta,k}$ for $k<n$ (and hence $\nu_{\eta,n}\parallel\nu_{\eta,k}$ for $k<n$),
        
        \item[--] if $k<n$ and $\rho_k\notin M_{\leq\eta}$, then $\rho_k \parallel
        \nu_{\eta,n}$.
    \end{enumerate}
    
    [Why possible? arriving to $n = m + 1,$ if for some $\nu \in \suc_{B}(\eta)$ we have $\nu \leq \rho_{ {m}}$ then clearly we can choose $\nu_{\eta, n},$ otherwise assume there is no $\nu$ as required, \underline{then} $\nu \in \suc_{B}(\eta) \setminus \{ \nu_{\eta, \ell}: \ell < m \} \Rightarrow \rho_{m} < \nu.$] 
    
    Next, by downward induction on $\eta\in B$ we define
    \[B_\eta=\bigcup\big\{B_{\nu_{\eta,n}}:n<\omega\big\}\cup\{\eta\}.\]
    Lastly we define ${\mathbf y}$ so that:\\
    $(M_{\mathbf y},<_{\mathbf y})=(M_{\mathbf x},<_{\mathbf x})$,\\
    $\cA^{\mathbf y}_\eta=\cA^{\mathbf x}_\eta$ if $\eta\in M_{\mathbf x}$
      but $\eta\notin B\setminus \max(B)$, and\\
    $\cA^{\mathbf y}_\eta=\cA^{\mathbf x}_\eta\cup\{B_\eta\}$ if $\eta\in B\setminus
    \max(B)$,\\ 
    $\cB_{\mathbf y}=\cB_{\mathbf x}\cup\{B_{\rt_{\mathbf x}}\}$ and for $B',B''\in\cB_{\mathbf y}$ we let:
        \begin{quote}
            $B'\leq_{\mathbf y} B''$ if and only if
            $B'\leq_{\mathbf x}B''$ or $B''= B_{\rt_{\mathbf x}}$.\qedhere
        \end{quote}
\end{PROOF}

\begin{lemma}\label{8h.15} 
    \ENUMplusONE\begin{enumerate}
    \item\label{815.1}
     If $\mathbf x \in \mathbf K_{\le \aleph_0}$, $Y \in \af_{\mathbf x}$
    and $Z \subseteq Y$ \then \, for some $\mathbf y\in\mathbf K_{\le \aleph_0}$
    we have $\mathbf x \le_{\mathbf K} \mathbf y$ and either $Z\in D^{\mathbf y}_Y$
    or $(Y \backslash Z) \in D^{\mathbf y}_Y$.
    \item \label{815.2}
     Moreover, if $h$ is a function with domain $Y$, \then \, above we
    can demand that for some $B \in \cB_{\mathbf y}$, $Y\cap B$ is a
    front of $B$ and for some front $Y'$ of $B$ which is below $Y$ and a
    one-to-one function $h'$ with domain $Y'$ we have
    \[\rho \in Y' \wedge \varrho \in Y \cap B \wedge \rho \le_{M_{\mathbf y}}
    \varrho \quad\Rightarrow\quad h(\rho) = h'(\varrho).\]
    (Note that possibly $Y' = \{\rt_{\mathbf y}\}$ and then $h\rest (Y \cap B)$ is
    constant.) 
    \end{enumerate}
\end{lemma}

\begin{PROOF}{\ref{815.1}}
    By Lemma~\ref{8h.10}(\ref{810.1}) \wilog \, there is $B\in\cB_{\mathbf x}$ such
    that $B$ is $\le_{\mathbf x}$-maximal in ${\cB}_{\mathbf x}$; clearly $Y \cap B$ is an almost front of $B$ and so \wilog \, $Y\subseteq B$.
    
    We know that $B[{\le}Y]:=\{\rho\in B:(\exists \nu)[\rho\le_{M_{\mathbf
        x}}\nu\in Y]\}$ has no $\omega$--branch, so by $<_{M_{\mathbf
        x}}$--downward induction on $\nu\in B[{\le}Y] $ we choose $({\mathbf
      t}_\nu,Y_\nu)$ such that (where $M=M_{\mathbf x}$, of course):
    \ENUMplusONE\begin{enumerate}
    \item[(a)]  $\mathbf t_\nu \in \{0,1\}$ and:\\
    \plusindent if $\mathbf t_\nu = 1$, then $Y_\nu \subseteq M_{\ge \nu} \cap
    Z$ , 
    \\
    \plusindent if $\mathbf t_\nu = 0$, then $Y_\nu \subseteq M_{\ge \nu}\cap
    (Y \backslash Z)$, 
    \item[(b)]  $Y_\nu=\max(B'_\nu)$ for some $B'_\nu\in\psb_M(B_{\ge\nu})$,
    \item[(c)]  if $\nu\in Y$ then $Y_\nu=\{\nu\}$ and $\mathbf t_\nu =$ (the truth value of $\nu \in Z$),
    \item[(d)]  if $\nu\in B[{\le}Y] \backslash Y$ then:
    \plusindent for every finite set $F\subseteq M\setminus M_{\leq \nu}$
      there are infinitely many $\varrho\in\suc_B(\nu)$ such that $(\forall\rho
      \in F)(\rho\parallel \varrho)$ and ${\mathbf t}_\varrho = {\mathbf t}_\nu$, 
    \plusindent       $Y_\nu = \bigcup\{Y_\varrho:\varrho \in \suc_B(\nu)$
    and $\mathbf t_\varrho = \mathbf t_\nu\}$.
    \end{enumerate}
    This is easily done and so $\mathbf t_{\rt_{\mathbf x}}$ is well defined.
    For $\nu\in B[{\le}Y]$ we let
    \[B^*_\nu=\{\rho \in B_{\ge\nu}:\mbox{ for some }\varrho \in Y_\nu
    \mbox{ we have }\varrho \le_M \rho \vee \rho \le_M\varrho\}.\]
    Now define $\mathbf y$ by adding $B^*_\nu$ to $\cA^{\mathbf x}_\nu$
    for every $\nu \in B[{\le}Y]$, and check.
\end{PROOF}

\begin{PROOF}{\ref{815.2}}
    First note that by Lemmas~\ref{8h.10}(\ref{810.1}) and~\ref{8h.11} we may assume that there is $B\in\cB_{\mathbf x}$ such that $B$ is $\leq_{\mathbf x}$--maximal, the set $Y$ is a front of~$B$, and:
    
    \begin{quote}
        if $\nu\in B\setminus \max(B)$ and $\rho\in M\setminus
        M_{\leq\nu}$, \\ then for all but finitely many $\varrho\in \suc_B(\nu)$ we
        have $\rho\parallel_M \varrho$.  
    \end{quote}
    
    Now note: if $h':Y'\longrightarrow A$, $Y'\in\frt(B')$, $Z = \{\eta\in B':
    \suc_{B'}(\eta)\subseteq Y'\}$ is a front of $B'$ and $h'\rest\suc_{B'}(\eta)$ is one-to-one for all $\eta\in Z$, \then \, we can
    find $B''\in\psb_M(B)$ such that $h' \rest B'' \cap Y'$ is one-to-one. So we
    may follow similarly as in (\ref{815.1}).
\end{PROOF}

Let us recall the following definition. 

\begin{definition}[P-points and Q-points]\label{z9}
    Let $D$ be a nonprincipal ultrafilter on a countable set $\Dom(D)$. 
    
    We say $D$ is a $Q$-point if: 
      whenever  $f$ is a finite-to-one function with
      domain $\Dom(D)$, then $f\rest A$ is one-to-one for some $A \in D$.
    
    We say that $D$ is a $P$-point if: 
      for each sequence $\langle A_n:n<\omega\rangle$ of sets from $D$
      there is an $A\in D$
      such that $A\setminus A_n$ is finite for each $n<\omega$. 
\end{definition}

We can conclude the main result of this section.

\begin{theorem}\label{8h.17}
    Assume CH. There is a $\mathbf x \in \mathbf K$ such that:
    \ENUMplusONE\begin{enumerate}[(a)]  
    \item\label{8h17a} 
    \ENUMplusONE\begin{enumerate}  
    \item[$(\alpha)$] ${\cA}^{\mathbf x}_\eta \ne \big\{\{\eta\}\big\}$
    for $\eta\in M_{\mathbf x}$,
    \item[$(\beta)$] ${\cB}_{\mathbf x} = \cA^{\mathbf x}_{\rt(\mathbf x)}
    \setminus\big\{\{\rt_{\mathbf x}\}\big\}$ is $\aleph_1$--directed under $\le_{\mathbf x}$,
    \end{enumerate}
    \item\label{8h17b} 
      if $Y\in\frt^-_{\mathbf x}$, then
    \ENUMplusONE\begin{enumerate}
    \item[$(\alpha)$] $D^{\mathbf x}_Y$ is a non-principal ultrafilter on $Y$, and 
    \item[$(\beta)$]  $D^{\mathbf x}_Y$ is a $Q$-point, see Definition~\ref{z9},
    \end{enumerate}
    \item\label{8h17c} 
      if $B_1\in\cB_{\mathbf x}$, \then \, for some $B_2\in
    \cB_{\mathbf x}$ we have $B_1\le_{\mathbf x} B_2$ and $B_1\cap
    \suc_{B_2}(\rt_{\mathbf x})=\emptyset$, moreover\footnote{Not a serious
    addition.  As always, the number of $\varrho\in \suc_{B_2}
    (\rt_{\mathbf x})$ failing this is finite.} 
    \[(\forall \varrho\in \suc_{B_2}(\rt_{\mathbf x}))(\exists^\infty
    \rho\in\suc_{B_1}(\rt_{\mathbf x}))[\varrho\le_{M_{\mathbf x}} \rho].\]
    \item\label{8h17d} 
      $\mathbf x$ is (see Definition~\ref{7g.14}): fat, big, large, and
      full. 
    \end{enumerate}
\end{theorem}

\begin{PROOF}{\ref{8h.17}}
    We choose ${\mathbf x}_\alpha\in {\mathbf K}_{\leq\aleph_0}$ by induction on $\alpha<\aleph_1$ so that
    \ENUMplusONE\begin{enumerate}
    \item[(i)]  if $\beta<\alpha<\aleph_1$, then ${\mathbf x}_\beta\leq_{\mathbf K}
      {\mathbf x}_\alpha$, 
    \item[(ii)] for each successor $\alpha$, there is a $\leq_{{\mathbf
          x}_\alpha}$--maximal element in $\cB_{{\mathbf x}_\alpha}$.
    \end{enumerate}
    We use a bookkeeping device to ensure largeness and bigness and
    \ITEMIZEplusONE\begin{itemize}
    \item[--] for $\alpha=0$ we use Lemma~\ref{8h.3},
    \item[--] for $\alpha$ limit we use Definition~\ref{7g.7}(\ref{77.3}) and Observation~\ref{7g.9}(\ref{79.2}), 
    \item[--] if $\alpha=\beta+1$, $\beta$ is limit, then we use Lemma~\ref{8h.15}(\ref{815.1}) (and the instructions from our bookkeeping device) to take
      care of the bigness,
    \item[--] if $\alpha=\beta+2$, $\beta$ is limit, then we use Lemma~\ref{8h.15}(\ref{815.2}) (and the instructions from our bookkeeping device) to take
      care of the largeness,
    \item[--] if $\alpha=\beta+3$, $\beta$ is limit, then we use Lemma~\ref{8h.10}(\ref{810.3},\ref{810.4}) (and the instructions from our bookkeeping device) to
      ensure that at the end $\mathbf x$ is fat and full,
    \item[--] if $\alpha=\beta+k$, $\beta$ is limit, $4\leq k<\omega$, then
    we ensure clause  (\ref{8h17d}).
    \end{itemize}
    In the end we let ${\mathbf x}=\bigcup\limits_{\alpha<\aleph_1} {\mathbf
      x}_\alpha$. Then ${\mathbf x}$ is fat, big, large and $\cB_{\mathbf x}$ is $\aleph_1$--directed. Note that clause (\ref{8h17b})$(\beta)$ follows from the
    largeness. 
\end{PROOF}

\begin{definition}\label{8j.20d}
    \ENUMplusONE\begin{enumerate}
    \item We say that ${\mathbf x}\in {\mathbf K}$ is {\em nice\/} if it satisfies
      conditions (\ref{8h17a})--(\ref{8h17d}) of Theorem~\ref{8h.17}. The class of all nice ${\mathbf
        x}$ is denoted by ${\mathbf K}_{\rm n}$.
    \item An ${\mathbf x}\in {\mathbf K}$ is {\em reasonable\/} if it satisfies (\ref{8h17a}),
      (\ref{8h17c}) of Theorem~\ref{8h.17}. Let ${\mathbf K}_{\rm r}$ be the set of all $\mathbf x \in \mathbf K$ which are reasonable.
    \item Let ${\mathbf K}_{\rm u}$ be the set of $\mathbf x \in \mathbf K_{\rm r}$ for
      which clause (\ref{8h17b})$(\alpha)$ of Theorem~\ref{8h.17} holds.
    \item For $\mathbf x \in \mathbf K$ we say that $\cI \subseteq \cA_{\mathbf x}$ (see Definition~\ref{7g.1}(\ref{71.b})) 
       is {\em $\mathbf x$--dense\/} iff:\\
     for every $B_1\in \cB_{\mathbf x}$ there is $B_2$ such that
    \ENUMplusONE\begin{enumerate}
    \item[$(\alpha)$] $B_1 \le_{\mathbf x} B_2 \in \cB_{\mathbf x}$, and
    \item[$(\beta)$] if $A\subseteq M_{\mathbf x}\setminus \{
    \rt_{\mathbf x}\}$ is finite, then for some $\nu$ we have
    \[\nu\in\suc_{B_2}(\rt_{\mathbf x}),\quad (B_2)_{\geq \nu}\in\cI,
    \quad\mbox{ and }\quad (\forall \rho\in A)(\rho\parallel \nu).\]
    \end{enumerate}
    \item For $\mathbf x \in \mathbf K$ we say $\cI$ is {\em $\mathbf x$--open\/}
    if $\cI \subseteq \cA_{\mathbf x}$  and
    \begin{quote}
      if $B_1 \in \cI$ then $\seb(B_1)\cap\cA_{\mathbf x}\subseteq\cI$.  
    \end{quote}
    \item Let ${\mathbf K}{\rm _g}$ be the class of $\mathbf x \in \mathbf K_{\rm r}$
      which are {\em good}, which means: if $\cI$ is $\mathbf x$--dense, $\mathbf
      x$--open and $B_1\in \cB_{\mathbf x}$ \then \, for some $B_2\in\cB_{\mathbf
        x}$ we have $B_1\le_{\mathbf x} B_2$ and $(B_2)_{\ge \eta}\in\cI$ for all
      but finitely many $\eta\in\suc_{B_2}(\rt_{\mathbf x})$.
    \item We say that  $\mathbf x \in \mathbf K$  is {\em ultra\/} if it is both nice and
    good. Let $\mathbf K_{\ut}$ be the class of $\mathbf x$ which are ultra, i.e., $\mathbf K_{\ut}=\mathbf K_{\rm g}\cap \mathbf K_{\rm n}$.
    \end{enumerate}
\end{definition}

\begin{theorem}\label{8j.28}
    Assume $\diamondsuit_{\aleph_1}$.   Then there exists an ultra $\mathbf x \in{\mathbf K}$. 
\end{theorem}

 %
 
\begin{PROOF}{\ref{8j.28}}
    We repeat the proof of Theorem~\ref{8h.17} but at limit stages $\delta<\aleph_1$ we use additionally $\diamondsuit_{\aleph_1}$ to take care of the additional demand $\mathbf x \in \mathbf K_{\rm g}$ here. 
    
    So we are given: a limit ordinal  $\delta<\aleph_1$ and a set $\cJ \subseteq  \cA_{\mathbf x_\delta}$ 
    such that for some ${\mathbf y}\in \mathbf K$ with $\mathbf x_\delta\le\mathbf y$ \ and some $\cI \subseteq \cA_{\mathbf y}$ we have
    
    \begin{quote}
        The set $\cI$ is dense open in $\cA_{\mathbf y}$, satisfies $\cJ=\cI\cap \cA_{{\mathbf x}_\delta}$, and moreover: There is a countable elementary submodel $N\prec\cH(\aleph_2)$  \\ with $({\mathbf y},\cI)\in N$ and $({\mathbf x}_\delta,\cJ)=({\mathbf y}\rest N, \cI\cap N)$, so $M_{{\mathbf x}_\delta}=M_{\mathbf y}\rest N$, etc.
    \end{quote}
    
    Let $\langle B^0_\ell:\ell<\omega\rangle$ be an increasing cofinal subset of $(\cB_{{\mathbf x}_\delta},\leq_{{\mathbf x}_\delta})$. For every $\ell$ there is $B^1_\ell\in \cB_{{\mathbf x}_\delta}$ such that $B^0_\ell
    \leq_{{\mathbf x}_\delta} B^1_\ell$, and for every finite $A\subseteq M_{{\mathbf x}_\delta}\setminus \{\rt({\mathbf x}_\delta)\}$ there is $\nu\in
    \suc_{B^1_\ell}(\rt( {\mathbf x}_\delta))$ such that
    \[(\forall \rho\in A)(\rho\parallel \nu)\quad \mbox{ and }\quad (B^1_\ell)_{\geq \nu}\in\cI.\] 
    Clearly, for every $\ell$ for some $k(\ell)>\ell$ we have $B^1_\ell\leq _{{\mathbf x}_\delta} B^0_{k(\ell)}$. We can choose $\langle
    \ell_n:n<\omega\rangle$ so that $k(\ell_n)<\ell_{n+1}$.  Let $B_n=B^1_{\ell_n}$. We continue as in Lemma~\ref{8h.10}(\ref{810.1}) using the $\langle B_n:n<\omega\rangle$ and, when choosing $\nu_n$, demanding additionally that $(B_n)_{\geq \nu_n}\in\cI$. (Note that $(B_n)_{\geq \nu_n}\in\cI$ implies $(B^*_n)_{\geq\nu_n}\in\cI$ for $B^*_n$ as there.)
\end{PROOF}

\begin{proposition}\label{8j.17}
    Assume $\mathbf x \in \mathbf K_{\rm n}$. 
    \ENUMplusONE\begin{enumerate}[(i)]
    \item \label{817.i}   If $B \in {\cB}_{\mathbf x}$ and $Y_1,Y_2 \in \frt(B)$ and $Y_2$ is above $Y_1$, \then \, $h^{\mathbf x}_{Y_2,Y_1}$ exemplifies $D^{\mathbf
        x}_{Y_1}\le_{\RK} D^{\mathbf x}_{Y_2}$. 
    \item \label{817.ii}   The family $\{D^{\mathbf x}_Y:Y \in\frt^-_{\mathbf x}\}$ is $\geq_{\RK}$--directed (even $\aleph_1$ directed). 
    \item \label{817.iii}  If $Y\in\af^-_{\mathbf x}$, then $\leq_{\RK}$--below $D^{\mathbf x}_Y$ there is no $P$-point. 
    \end{enumerate}
\end{proposition}

\begin{PROOF}{\ref{8j.17}}
    (\ref{817.i})\quad  Follows from Observation~\ref{7g.5g}(\ref{75.6}).
    
    \noindent (\ref{817.ii})\quad By (\ref{817.i})  and the directedness of $\cB_{\mathbf
      x}$. 
    
    \noindent (\ref{817.iii})  \quad Let $B_1 \in \cB_{\mathbf x}$ be such that $B_1 \cap
    Y$ is an almost front of~$B_1$. Suppose that $h:Y\longrightarrow \bbN$ is
    such that $h^{-1}[\{n\}]=\emptyset \mod D^{\mathbf x}_Y$ for every $n$, hence
    there is $A_n \in \cB_{\mathbf x}$ which witnesses this. Assume towards
    contradiction that $h(D^{\mathbf x}_Y)$ is a $P$-point; \wilog \, $h$ is onto $\bbN$.
    As $\cB_{\mathbf x}$ is $\aleph_1$--directed we may pick $B_2\in\cB_{\mathbf x}$
    such that $A_n \le_{\mathbf x} B_2$ (for all $n<\omega$) and $B_1 \le_{\mathbf
      x} B_2$. 
    
    As $\mathbf x$ is large, we may apply the Definition~\ref{7g.14} of large to
    the pair $(B_2,h')$ where $h'(\eta) = h(\nu)$ when $\nu
    \le_{M_{\mathbf x}}\eta\in\max(B)$ and zero if there is no such~$\nu$. So
    there are $B_3,Y_3$ such that
    \begin{itemize}
    \item[--] $B_2\leq_{\mathbf x} B_3$,
    \item[--] $Y_3$ is a front of $B_3$  below $Y\cap B_3$,
    \item[--] for $\eta,\nu \in Y\cap B_3$ we have:\quad $h(\eta) =
      h(\nu)\ \Leftrightarrow\ (\exists \rho \in Y_3)(\rho\le_{M_{\mathbf
          x}}\eta\wedge\rho \le_{M_{\mathbf x}}\nu)$.
    \end{itemize}
    Let $Z=\suc_{B_3}(\rt_{\mathbf x})$. If $Y_3=\{\rt_{\mathbf x}\}$,
    then for some $n$ we have $h^{-1}[\{n\}]\in D^{\mathbf x}_Y$, a
    contradiction. Therefore $Y_3\neq\{\rt_{\mathbf x}\}$ and thus $\rt_{\mathbf x}\notin
    Y_3$, so $Y_3$ is above~$Z$. Clearly, $D^{\mathbf x}_Z\le_{\RK} h(D^{\mathbf
      x}_Y)$ and  hence $D^{\mathbf x}_Z$ is a $P$-point.
    
    By clauses (\ref{8h17c}) and (\ref{8h17d}) of Theorem~\ref{8h.17} there is $B_4\in\cB_{\mathbf x}$
    such  that $B_3\le_{\mathbf x} B_4$, $B_4\cap Z$ is a front of $B_4$  and
    \[(\forall\varrho\in\suc_{B_4}(\rt_{\mathbf x}))(\exists^\infty \rho
    \in\suc_{B_3}(\rt_{\mathbf x}))[\varrho\le_{M_{\mathbf x}}\rho].\]
    For each $\varrho\in\suc_{B_4}(\rt_{\mathbf x})$ let $Z_\varrho = \{\rho \in
    Z:\varrho \le_{M_{\mathbf x}} \rho\}$, so $\langle Z_\varrho:\rho \in
    \suc_{B_4}(\rt_{\mathbf x})\rangle$ is a partition of $Z$, and $Z_\varrho=
    \emptyset \mod D^{\mathbf x}_Z$ for each $\varrho$. But clearly there is no $Z' \in D^{\mathbf x}_Z$ such that $Z' \cap Z_\varrho$ is finite for every $\varrho \in \suc_{B_4}(\rt_{\mathbf x})$, contradiction to ``$D^{\mathbf x}_Z$
    is a $P$-point''. 
\end{PROOF}

\newpage

\section{Basic connections to forcing}

\begin{definition}\label{k3}
    For a forcing notion $\bbQ$ and $p \in \bbQ$ we define $ \Game^{\seb}_p = \Game^{\seb}_{\bbQ,p}$, the strong bounding game between the null player NU and the bounding player BND as follows:
    \begin{itemize}
    \item[--] A play last $\omega$ moves, and
    \item[--] in the $n$-th move:
    \ENUMplusONE\begin{enumerate}
    \item[$*$]  first the NU player gives a (non-empty) tree $\cT_n$ with $\omega$ levels and no maximal node and a $\bbQ$-name $\name F_n$ of a
      function with domain $\cT_n$ such that 
    \[\eta\in\cT_n\quad \Rightarrow\quad p\Vdash_{\bbQ}\mbox{`` }
    \name{F}_n(\eta) \in \suc_{\cT_n}(\eta)\mbox{ ''},\]
    \item[$*$]   then BND player chooses $\eta_n \in \cT_n$.
    \end{enumerate}
    \item[--] In the end, the BND player wins the play $\langle \cT_n, \eta_n:
      n<\omega \rangle$ \Iff \,   there is $q \in \bbQ$ above $p$ forcing that 
    \[\big(\forall n<\omega\big)\big(\exists k< \mbox{level}
    (\eta_n)\big) \big(\name F_n(\eta_n\rest k)\le_{\cT_n}\eta_n\wedge k\mbox{
      is even }\big),\]
    where $\eta_n\rest k$ is the unique $\nu\le_{\cT_n}\eta_n$ of level~$k$.
    \end{itemize}
    
    Omitting $p$ means NU chooses it in his first move.
    The game  $ \Game^{\seb}_{\bbQ}$ (without a parameter $p\in \bbQ$ is defined
    similarly, but here the first player NU also chooses a condition $p$ in 
    the first move. 
\end{definition}

\begin{definition}\label{k3x}
    A forcing notion $\bbQ$ is {\em strongly bounding\/} if for every
    condition $p\in\bbQ$ player BND has a winning strategy in the game $\Game^{\seb}_{\bbQ,p}$.
    \end{definition}
    \begin{definition}
    \label{7g.18} 
    \ENUMplusONE\begin{enumerate}
    \item\label{718.1} We say $\cP \subseteq [\bbN]^{\aleph_0}$ is big iff: for every $\mathbf c:\bbN \rightarrow \{0,1\}$ there is $A \in \cP$ such that $\mathbf c \rest A$ is constant.
    \item\label{718.2} For $B \in \CTW({}^{\omega >}\omega,\triangleleft)$ we say that
    a family $\cB \subseteq\psb(B)$ is big (in $B$) iff: for every $\mathbf c:\max(B)\longrightarrow \{0,1\}$ there is $B' \in \cB$ such
    that $\mathbf c \rest\max(B')$ is constant.
    \item\label{718.3} For $B \in \CTW({}^{\omega >}\omega,\triangleleft)$ we say that a
    family  $\cB \subseteq\psb(B)$ is large (in $B$) iff\\
        for every function $\mathbf c$ with domain $\max(B)$
        there is $B' \in \cB$ and front $Y$ of $B'$
          such that 
    \begin{quote}
       for every $\eta,\nu \in \max(B')$ we have
      \begin{quote}
     $\mathbf c(\eta)= \mathbf c(\nu)\ \Leftrightarrow\ (\exists \rho \in Y)(\rho
    \le_B \nu\ \wedge\ \rho \le_B \eta)$.
      \end{quote}
    \end{quote}
    \end{enumerate}
\end{definition}

\begin{theorem}\label{k2}
    Assume that:
    \ENUMplusONE\begin{enumerate}[(a)] 
    \item\label{k2a} $B\in\CWT(M)$ for a partial order $M$, \wilog\,
      $M=({}^{\omega>}\omega,\triangleleft)$, 
    \item\label{k2b} The forcing notion  $\bbQ$ is  strongly bounding.
    \item\label{k2c}
    \ENUMplusONE\begin{enumerate}
    \item[$(\alpha)$] forcing with $\bbQ$ preserves some
    non-principal ultrafilter on $\bbN$, \BOLD{or just}
    \item[$(\beta)$] $([\bbN]^{\aleph_0})^{\mathbf V}$ is big in $\mathbf V^{\bbQ}$,
      see Definition~\ref{7g.18}, 
    \end{enumerate}
    \item\label{k2d} $p \Vdash ``\name A \subseteq \max(B)$''.
    \end{enumerate}
    Then there are $B'\in \psb(B)$ and $q\in\bbQ$ such that $p \le q$ and 
    \[q \Vdash\mbox{ `` }\max(B')\subseteq \name \tau\mbox{ ''\quad
      \BOLD{or}\quad } q \Vdash\mbox{ `` }\max(B') \subseteq\max(B)
    \backslash\name\tau\mbox{ ''.}\]  
\end{theorem}

\begin{PROOF}{\ref{k2}} 
    We prove this by induction on $\Dp(B)$ (see Definition~\ref{a5-56}), for all
    such $B$'s. Let $\eta = \rt(B)$.

    \noindent \BOLD{Case 1}:  $\Dp(B) = 0$
    
    Trivial, as then $B = \{\eta\}$, i.e., $B$ is a singleton so $B' = B$ can
    serve. 
    
    \noindent
    \BOLD{Case 2}:  Dp$_{\mathbf x}(B) = 1$
    
    Then $\Dp(B_{\geq\nu})=0$ for all $\nu\in B \backslash \{\eta\}$.  Now, $|B
    \backslash \{\eta\}|=\aleph_0$ and we just need to find $p' \in \bbQ$ above $p$ such that $\{\nu \in B:\nu \ne \eta$ and $p'$ forces $\nu \in \name A$
    or forces $\nu \notin \name A\}$ is infinite.  As $\Vdash_{\bbQ}$ ``$\big([\bbN]^{\aleph_0}\big)^{\mathbf V}$ is big in ${\mathbf V}^{\bbQ}$'' (see
    clause  (\ref{k2c}) of our assumptions) this is possible.

    \noindent \BOLD{Case 3}:  $\alpha=\Dp(B) > 1$
    
    Let $Y = \suc_{B}(\eta)$. Then for $\nu \in Y$ we have $\Dp(B_{\geq\nu})<
    \alpha$, hence the induction hypothesis applies to $B_{\geq\nu}$. We may
    assume that if $\rho$ is not below $\eta$ then for all but finitely many $\nu\in Y$ we have $\nu\parallel \rho$ (cf.~the proof of Lemma~\ref{8h.11}).
    Let $\langle \nu_n:\nu \in \bbN\rangle$ list~$Y$. 
    
    We simulate a play of $\Game^{\seb}_{\bbQ,p}$ in which the BND player uses a
    winning strategy and the NU player acts so that in the $n$-th move:
    \begin{itemize}
    \item[--] $\cT_n = \big\{\langle B_0,\dotsc,B_{k-1}\rangle:k
    \in \bbN$, $B_\ell \in \psb(B_{\geq\nu_n})$ for $\ell < k$ and $B_{\ell+1}
    \subseteq B_\ell$ if $\ell +1 < k\big\}$, 
    \item[--] the relation $<_{\cT_n}$ is being an initial segment,
    \item[--] $\name F_n(\langle B_0,\dotsc,B_{k-1}\rangle)$ is $\langle B_0,\dotsc,B_{k-1},B'\rangle$ for some $B'\in \psb(B_{k-1})\cap
    {\mathbf V}$ such that 
    \begin{center}
    either\quad $\max(B')\subseteq \name A$\quad or\quad $\max(B')\cap \name
    A=\emptyset$.  
    \end{center}
    \end{itemize}
    There is such a function $\name F_n$ because of the induction hypothesis.
    
    Clearly we can do this.  As the player BND has used a winning strategy, BND
    has won the play so there is $q \in \bbQ$ stronger than $p$ and such that
      $q \Vdash$ ``for every $n$ for some even $k <
    \text{ level}_{\cT_n}(\eta_n)$ we have $\name F_n(\eta_n \rest k) \le_{\cT_n} \eta_n$".
    
    Hence by the choice of $(\cT_n,\name F_n)$, letting $\eta_n = \langle
    B_{n,0},\dotsc,B_{n,k(n)}\rangle$ we have:\\
      for some $\langle\name{\mathbf t}_n:n\in\bbN\rangle$
    \begin{itemize}
    \item[--]  $B_{n,k(n)} \in \psb(B_{\geq\nu_n})$,
    \item[--]  $\name{\mathbf t}_n$ is a $\bbQ$-name of the truth
    value,
    \item[--]  $q \Vdash$ ``if $\name{\mathbf t}_n = 1$,
      then $\max(B_{n,k(n)}) \subseteq \name A$,
    \item[--]  
       if $\name{\mathbf t}_n = 0$ then $\max(B_{n,k(n)})\cap \name A = \emptyset$".
    \end{itemize}
     Now by clause (\ref{k2c})  of our assumptions
    \begin{quote}
     there is an infinite $\cU \subseteq \bbN$, a truth value $\mathbf t$ and a condition $r$\\ such that $q \le_{\bbQ} r$ and $r \Vdash
    ``\name{\mathbf t}_n = \mathbf t$ for $n \in \cU$''. 
    \end{quote}
    Lastly, let $B_* = \bigcup\{B_{n,k(n)}:n \in \cU\} \cup \{\eta\}$ and
    clearly $B_*,r$ are as required.
\end{PROOF}

\begin{remark}\label{k2-d}
    In the assumption (\ref{k2b}) of Theorem~\ref{k2} it is enough that the BND player
    does not lose the game $\Game^{\seb}_{\bbQ}$, i.e., the NU player has no
    winning strategy. 
\end{remark}

\begin{theorem}\label{k2m}
    Assume that 
    \ENUMplusONE\begin{enumerate}
    \item[(a)] $\bbQ$ is an ${}^\omega\omega$-bounding proper forcing notion, 
    \item[(b)] forcing with $\bbQ$ preserves some $P$-point, and 
    \item[(c)]  $B \in \CTW({}^{\omega >} \omega,\triangleleft)$.
    \end{enumerate}
    Then $(\psb(B))^{\mathbf V}$ is big in $\mathbf V^{\bbQ}$; see Definition
    \ref{7g.18}(\ref{718.2}).
\end{theorem}

\begin{PROOF}{\ref{k2m}}
    Let $D$ be a $P$-point ultrafilter such that $\Vdash_{\bbQ}$`` $D$ generates an ultrafilter '' and $p\in\bbQ$. Suppose that $p\Vdash$``$\name{c}:\max(\name{B}) \longrightarrow\{0,1\}$ ''. Let $\chi$ be a large
    enough regular cardinal and $N\prec (\cH(\chi),{\in})$ be a countable model
    with $B,\bbQ,p, \name{c},\ldots\in N$. Let $q\in\bbQ$ be such that:
    
    \ENUMplusONE\begin{itemize}
    \item[--] $p\leq_{\bbQ} q$,
    
    \item[--] $q$ is $(N,\bbQ)$--generic,
    
    \item[--] for some $g\in \big({}^\omega\omega\big)^{\mathbf V}$ we have $q\Vdash$`` if $\name{f}\in {}^\omega\omega\cap N$, then $\name{f} <_{J^{\bd}_\omega} g$ '',
    
    \item[--] for some $A\in D$ we have $q\Vdash$`` if $\name{B}\in D\cap
    N$, then $A\subseteq^* \name{B}$ ''.
    \end{itemize}
    \relax From $(g,A)$ we can compute ${\mathbf c}$ and $B'\in\big(\psb(B)
    \big)^{\mathbf V}$ such that $q\Vdash$`` $\name{c}\rest B'$ is constantly ${\mathbf c}$ '', so we are done.
\end{PROOF}

\begin{theorem}\label{k12}
    Assume that ${\mathbf x}\in {\mathbf K}$ and
    \ENUMplusONE\begin{enumerate}[(A)]
    \item\label{k12A} The forcing notion  $\bbQ$ is a proper forcing notion,
    \item\label{k12B} the set  $D_*$ is a Ramsey ultrafilter in $\mathbf V$,
    \item\label{k12C}  $\Vdash_{\bbQ} ``\fil(D_*)$ is a Ramsey ultrafilter'',
    \item\label{k12D}  $B\in\cB_{\mathbf x}$.
    \end{enumerate}
    \Then\, $(\psb(B))^{\mathbf V}$ is large in $\mathbf V^{\bbQ}$ (see Definition~\ref{7g.18}).
\end{theorem}

\begin{PROOF}{\ref{k12}}
    We prove this by induction on $\Dp(B)$ for $B\in \cB_{\mathbf x}$. Let $\mathbf
    c: \max(B)\longrightarrow \bbN$ be from $\mathbf V^{\bbQ}$ and we should find $(B',Y)$ as promised.  We shall work in ${\mathbf V}^{\bbQ}$.
    
    If $\Dp(B)=0$, i.e., $|B|=1$ this is trivial.
    
    If $\Dp(B)=1$ let $\langle \eta_n:\eta \in \bbN\rangle \in \mathbf V$ list $\suc_B(\rt_{\mathbf x})$: by assumption (\ref{k12C}) in $\mathbf V^{\bbQ}$, for some $A\in\fil(D_*)$ the sequence $\langle \mathbf c(\eta_n):n \in A\rangle$ is
    constant or without repetitions. Without loss of generality $A \in D_*
    \subseteq \mathbf V$ and then $\{\rt_{\mathbf x}\} \cup \{\eta_n:n \in A\}$ is
    as required.
    
    So assume $\Dp(B) > 1$.  \Wilog \, $0 \notin \Rang(\mathbf c)$.  For $\nu \in B \backslash \max(B)$ let $\langle \eta_{\nu,n}:n \in
    \bbN\rangle$ list $\suc_B(\nu)$ so that the function $(\nu,n)\mapsto
    \eta_{\nu,n}$ belongs to $\mathbf V$.  In ${\mathbf V}^{\bbQ}$, by
    downward induction on $\nu \in B$, we choose $k_\nu = k(\nu)$, $A_\nu,
    A_{\nu,\rho}$ and ${\mathbf t}_{\nu,\rho}$ so that the following requirements
    (\ref{k12pa})--(\ref{k12pd}) are satisfied: 
    \begin{enumerate}[(a)]
    \item\label{k12pa} $k_\nu\in \bbN$, $A_\nu\in D_*$,
    \item\label{k12pb} if $\nu\in\max(B)$, then $k_n={\mathbf c}(\nu)$, so $>0$,  
    \item\label{k12pc} if $\nu \notin \max(B)$ then $(\alpha)_\nu$ or $(\beta)_\nu$
      where: 
    \ENUMplusONE\begin{enumerate}
    \item[$(\alpha)_\nu$] $k_\nu = 0$ and $\langle k(\eta_{\nu,n}):n \in
      A_\nu\rangle$ is with no repetitions, all non-zero,
    \item[$(\beta)_\nu$] $\langle k(\eta_{\nu,n}):n \in A_\nu\rangle$ is constantly $k_\nu$,
    \end{enumerate}
    \item\label{k12pd}  for $\nu,\rho \in B \backslash \max(B)$ we have $A_{\nu,\rho}
      \in D_*$ and ${\mathbf t}_{\nu,\rho}\in\{0,1\}$ and 
    
      either ${\mathbf t}_{\nu,\rho}=1$ and $n \in A_{\nu,\rho} \Rightarrow
      k(\eta_{\rho,n}) = k(\eta_{\nu,n})$
    
    or ${\mathbf t}_{\nu,\rho}=0$ and $\{k(\eta_{\rho,n}):n\in
    A_{\nu,\ell}\}$ is disjoint to $\{k(\eta_{\nu,n}):n\in A_{\nu,\rho}\}$.
    \end{enumerate}
    This is possible by assumption (\ref{k12C}). By the same assumption, there is $A_*
    \in D_*$ such that:
    \begin{itemize}
    \item[] if $\nu \in B \backslash \max(B)$ then  $A_*\subseteq^*
    A_\nu$,
    \item[] if $\nu,\rho\in B \backslash \max(B)$ then $A_*\subseteq^*
      A_{\nu,\rho}$.  
    \end{itemize}
    Let $\langle \nu_n:n \in \bbN\rangle$ list $B \backslash \max(B)$ and
    let $f_1$ be the function with domain $B \backslash \max(B)$ such that
    \[f_1(\nu) = \{\eta_{\nu,n}:n \in A_*\backslash A_\nu  \mbox{ or for some $k<\ell$ we have }\nu=\nu_\ell\ \wedge\ n\in A_*\setminus
    A_{\nu_k,\nu_\ell} \}\]
    (so  $f_1(\nu)\in [\suc_B(\nu)]^{< \aleph_0}$).
    
    As the forcing $\bbQ$ satisfies (\ref{k12C}), it is bounding, so there is a
    function $f_2 \in \mathbf V$ with domain $B \backslash \max(B)$ such that $f_1(\nu)\subseteq f_2(\nu) \in [\suc_B(\nu)]^{< \aleph_0}$. Clearly,
    letting 
    \begin{align*}
    B_1 := A_{B,f} := \big\{\,\nu \in B:   \ &   \text{{\bf if} $\rho \in B$
     satisfies $\rt_{\mathbf x} \le_B \rho <_B \nu$}\\
      &\text{and $n$ is such that $\eta_{\rho,n}
      \le_B \nu$,}\\
      & \text{{\bf then} $n \in A_*$ but $\eta_{\rho,n} \notin f_2(\nu)$}\,\big\}
    \end{align*}
    we have $B_1 \in \psb(B)^{\mathbf V}$. 
    
    Define
     \[ Y:= \{\,\nu \in B_1:k_\nu \ne 0 \text{ and } \rho <_B \nu
      \Rightarrow k_\rho = 0\,\}.\]
    Plainly, 
    \begin{quote}
     the set $Y$ is a front of $B_1$,  \\
     and if $\nu \in Y$ then $\mathbf c \rest(B_1)_{\ge \nu}$ is constantly $k_\nu$. 
    \end{quote}
    Note that 
    \begin{quote}
      if $\nu\in B_1$ and $k_\nu=0$, then either $k_\eta=0$ for
      all $\eta\in\suc_{B_1}(\nu)$, \\or $k_\eta>0$ for all $\eta\in
      \suc_{B_1}(\nu)$. 
    \end{quote}
    Hence:
    \begin{quote}
      if $\nu \in B_1 \backslash \max(B_1)$ and $\suc_{B_1}(\nu)$ is not
      disjoint to $Y$, \\then $\suc_{B_1}(\nu) \subseteq Y$.
    \end{quote}
    If $Y = \{\rt_{\mathbf x}\}$ we are done, so assume not. Let $Z = \{\eta \in
    B_1:\eta \notin \max(B_1)$ and  $\suc_{B_1}(\eta) \subseteq Y\}$.  So
    
    \begin{quote}
     both $Z$ and $Y$ are fronts of $B_1$,\\
     both $Z$ and $Y$ belong to $\mathbf V$,\\
     if $\nu\in Y$ then $\langle k_\rho:\rho\in \max \big((B_1)_{\geq
        \nu} \big)\rangle$ is constantly $k_\nu$.
    \end{quote}
    Also if $Z=\{\rt_{\mathbf x}\}$ we are done, so assume not. Let $\langle
    \nu_n:n \in \bbN\rangle$ list~$Z$.  As $\fil(D_*)$ is a Ramsey ultrafilter
    we can find $\bar n$ such that
    \begin{enumerate}
    \item[--] $\bar n =\langle n(i):i\in\bbN\rangle$ is an increasing
      enumeration of a member of $D_*$, hence $\bar{n}\in
      \mathbf V$,
    \item[--] if $\ell\leq i$ then $\eta_{\nu_\ell,n(i)}\in B_1$, 
    \item[--] if $\ell<i$, ${\mathbf t}_{\nu_\ell,\nu_i}=0$ and $\nu_\ell,\nu_i\in B_1[{\leq} Z]$, then $\{k(\eta_{\nu_i,n(j)}):
    i\leq j\}$ is disjoint from $\{k(\eta_{\nu_\ell,n(j)}):i\le j\}$,
    moreover it is disjoint from $\{k(\eta_{\nu_\ell,n(j)}:j\in \bbN\}$. 
    \end{enumerate}
    Lastly, as $\bar{n}\in \mathbf V$ we can find in $\mathbf V$ a partition $\langle C_\ell:\ell \in \bbN\rangle$ of $\bbN$ to (pairwise disjoint)
    infinite sets and let
    \[\begin{array}{ll}
    B_2 = \{\varrho\in B_1:&\mbox{if }\nu_\ell <_{B_1}\varrho\mbox{ and }
    \nu_\ell\in B_1[{\leq}Z],\\
    &\mbox{then for some }i\in C_\ell\mbox{ we have }
    i > \ell \mbox{ and }\eta_{\nu_\ell,n(i)} \le_{B_2} \varrho\}.
    \end{array}\]
    Easily $B_2 \in \mathbf V$, $B_2\in \psb(B_1)$ and it is as required.
\end{PROOF}

Motivated by Definition~\ref{k3} we introduce the following bounding games for a forcing notion $\bbQ$. 

\begin{definition}\label{k11}
    Let $\bbQ$ be a forcing notion and $p \in \bbQ$.  We will define 3 games: 
    \relax $\Game^{\bd}_p = \Game^{\bd}_{\bbQ,p}$, $ \Game^{\ufbd}_p = \Game^{\ufbd}_{\bbQ,p}$, and $ \Game^{\vfbd}_p = \Game^{\vfbd}_{\bbQ,p}$.  Each of the games 
    lasts $\omega$ rounds, and in each round player NU moves first, and player BND second. 
    
    The games $\Game^{\bd},  \Game^{\ufbd}, \Game^{\vfbd}$ are defined analogously, 
    but here the condition $p$ will be chosen by player NU in his first move.
    
    \ENUMplusONE\begin{enumerate}
    \item In the $n$-th round of  the game  $\Game^{\bd}_p$, 
    first the NU player gives a $\bbQ$-name $\name \tau_n$ of a member of $\mathbf V$ and then
    the BND player gives a finite set $w_n \subseteq
    \mathbf V$.   \\  After $\omega$ rounds, 
     the BND player wins the play
    iff  there is $q \in \bbQ$ above $p$ forcing ``$\name\tau_n \in
    w_n"$ for every~$n$.
    
    \item 
    In the $n$-th round of  the game $\Game^{\ufbd} _p$,  
    first the NU player chooses an ultrafilter $E_n$ on
    some set $I_n$ from $\mathbf V$ and a $\bbQ$-name $\name E^+_n$ of an
    ultrafilter on $I_n$ extending $E_n$ and a $\bbQ$-name $\name X_n$ of
    a member of $\name E^+_n$;  then 
    the BND player chooses $t_n \in I_n$.
    \\
     In the end of the play the BND player wins the play
    iff  there is $q \in \bbQ$ above $p$ forcing
    ``$t_n \in \name X_n$'' for every~$n$.
    \item The game $\Game^{\vfbd}_p $ is similar to $\Game^{\ufbd}_p$,  but now we demand
    \[\Vdash_{\bbQ}\mbox{`` }\name X_n \in E_n\mbox{ or just includes a
    member of }E_n\mbox{ '',}\]
    so $\name E^+_n$ is redundant.
    \end{enumerate}
\end{definition}

Basic relations between the games introduced above are given by the following result. 

\begin{proposition}\label{k14}
    Let $\bbQ$ be a forcing notion.
    \ENUMplusONE\begin{enumerate}
    \item\label{k14.1} If BND wins in $\Game^{\seb}_{\bbQ,p}$ \then \, BND wins in $\Game^{\bd}_{\bbQ,p}$ which implies that $\bbQ$ is a bounding forcing.
    \item\label{k14.2} The player BND wins in $\Game^{\bd}_{\bbQ,p}$ \Iff \, BND wins
    in $\Game^{\vfbd}_{\bbQ,p}$.
    \item\label{k14.3} If the player BND wins in $\Game^{\ufbd}_{\bbQ,p}$ \then \, BND
       wins in $\Game^{\vfbd}_{\bbQ,p}$.
    \item\label{k14.4} We can replace in (\ref{k14.1})--(\ref{k14.3}) above ``wins'' by ``does not lose''.
    \end{enumerate}
\end{proposition}

\begin{PROOF}{\ref{k14}}
    (\ref{k14.1})\quad The second implication is obvious, so we concentrate on the
    first. For every $\name\tau$, a $\bbQ$-name of an ordinal we define
    a pair $(T_{\name\tau},\name F_{\name\tau})$ as follows:
    
    \begin{itemize}
    \item[--] let $u = \{\alpha:\ \nVdash_{\bbQ} ``\name\tau \ne
      \alpha"\}$, it is a non-empty set of $\leq|\bbQ|$ ordinals,
    \item[--] $T_{\name\tau}$ is the tree $\{\eta:\eta
    \in {}^{\omega >}u\}$, i.e., ordered by $\triangleleft$ (being an
    initial segment),
    \item[--] $\name F_{\name\tau}(\eta)= \eta \char 94 \langle \name
      \tau \rangle$ for $\eta \in T_{\name\tau}$.
    \end{itemize}
    Clearly,
    \begin{itemize}
    \item[--] $T_{\name\tau}$ is in $\mathbf V$, a tree with $\omega$ levels,
    \item[--] $\name F_{\name\tau}$ is a $\bbQ$-name of a
      function with domain $T_{{\name \tau}}$ such that $\Vdash_{\bbQ} ``\name
      F_{\name\tau}(\eta)\in\suc_{T_{\name\tau}}(\eta)"$.
    \item[--]  if $q \in \bbQ$ and $\eta \in T_{\name\tau}$ (so $\Rang(\eta)$ is a finite subset of $u$) then the following are
      equivalent:
    \ENUMplusONE\begin{enumerate}
    \item[(i)]  $q \Vdash ``\name\tau \in \Rang(\eta)$'',
    \item[(ii)]  $q \Vdash$ ``for some $\nu \triangleleft \eta$ we have $\nu
      \char 94 \langle \name F_{\name\tau}(\nu)\rangle\trianglelefteq \eta$''.  
    \end{enumerate}
    \end{itemize}
    So playing the game $\Game^{\bd}_{\bbQ,p}$ we can ``translate'' it to a play
    of $\Game^{\seb}_{\bbQ,p}$ replacing the NU choice of $\name\tau_n$ by the
    choice of $(T_{\name\tau},\name F_{\name\tau})$.  Thus every strategy {\bf
      st}$_1$ of BND in $\Game^{\seb}_{\bbQ,p}$ translates it to a strategy {\bf
      st}$_2$ of the player BND in $\Game^{\bd}_{\bbQ,p}$.

    \noindent (\ref{k14.2})\quad We now need two translations.

    \noindent \BOLD{Translating $\Game^{\vfbd}_{\bbQ,p}$ to $\Game^{\bd}_{\bbQ,p}$}:
    
    So we are given a move $y = (I,E,\name X)$ of NU in a play of $\Game^{\vfbd}_{\bbQ,p}$ as in Definition~\ref{k11}, i.e.,
    
    \begin{itemize}
    \item[--] $I \in \mathbf V$, $E$ is an ultrafilter on $I$, in $\mathbf V$, and
    \item[--] $\Vdash_{\bbQ} ``\name X \in E$ or just includes a member $\name X'$ of $E$''. 
    \end{itemize}
    Now we have: \\
      if $q \Vdash ``\name X' \in \cW"$ where $\cW \subseteq E$ is
      finite ($\cW$ an object in $\mathbf V$ not a name), \\ 
       then $\bigcap\{A:A \in \cW\}$ is non-empty and $t \in \bigcap\{A:A\in\cW\}
      \Rightarrow q \Vdash ``t \in \name X' \subseteq \name X"$.
    
    \noindent\BOLD{Translating $\Game^{\bd}_{\bbQ,p}$ to $\Game^{\vfbd}_{\bbQ,p}$}:
    
    Given $y = (I,\name\tau),\name\tau$ a $\bbQ$-name of a member $I$ of $\mathbf
    V$ we define $I_y = [I]^{< \aleph_0} \in \mathbf V$ and choose $E_y \in \mathbf
    V$ an ultrafilter on $I_y$ such that $u_* \in [I]^{<\aleph_0} \Rightarrow
    \{u \in [I]^{< \aleph_0}:u_* \subseteq u\} \in E$; lastly we choose
    \[\name X_y = \{u \in [I]^{< \aleph_0}:\name\tau \in u\}.\]
    So $(I_y,E_y,\name X_y)$ is a legal move in $\Game^{\vfbd}_{\bbQ,p}$ and for
    a finite subset $t$ of $I$: 
    \begin{quote}
     if $q \Vdash ``t \in \name{X}_y"$ then $q \Vdash
      ``\name\tau \in t$''. 
    \end{quote}
    
    \noindent (\ref{k14.3})\quad Obvious.
    
    \noindent (\ref{k14.4})\quad The same proof.
\end{PROOF}

\begin{claim}\label{k17}
    (1) [CH] Let $\mathbb{Q}$ be a bounding Suslin-proper forcing preserving some $P$-point (or less as in \ref{k2}(c)) (see Judah-Shelah \cite{Sh:292} e.g. Sacks forcing or see Roslanowski-Shelah, \cite{Sh:470}). Then there is $\bfx$ as in \S B (so ultra) such that in addition:
    
    \begin{enumerate}
        \item[$(\ast)_{1}$] $D_{\bfx}$ generates an ultra filter in $\mathbf{V}^{\mathbb{Q}}.$  
    \end{enumerate}
    
    (2) If $\mathbb{Q}_{i}$ (for $i < i_{\ast} \leq \omega_{1}$) is a bounding Suslin-proper forcing notion, \underline{then} we can find $\mathbf{x}$ such that: 
    
    \begin{enumerate}
        \item[$(\ast)$] $D_{\mathbf{x}}$ generates an ultra filter in $\mathbf{V}^{\mathbb{Q}_{i}}$ for each i.
    \end{enumerate}
    
    (3) Let $\mathbb{Q}_{i}^{r}$ be a bounding Suslin-proper forcing with any real parameters $r$ each such forcing, preserving some $P$-point (for $i < i_{\ast} \leq \omega_{1}$). Let $\mathbb{P}$ be the limit of a CS-iteration of cases of $\mathbb{Q}_{i}^{r}.$ Then we can find $\mathbf{x}$ as above for $\mathbb{P}.$  
\end{claim}

\begin{PROOF}{\ref{k17}}
    (1) Choosing as before $\mathbf{x}_{\alpha} \in \mathbf{K}_{\leq  {x_{0}}}$ by induction on $\alpha < \omega_{1},$ in stage $\alpha.$ Let $\mathbf{A}_{\alpha}$ be the set of objects $\mathbf{a}$ consisting of (so $p = p_{\mathbf{a}}$, etc):
    
    \begin{enumerate}
        \item[$(\ast)_{\mathbf{a}}^{1}$] 
        
        \begin{enumerate}
            \item[(a)] $p \in \mathbb{Q},$
            
            \item[(b)] $B \in \mathbf{A}_{Y}$ where $Y$ is a front of $B,$
            
            \item[(c)] $p_{\eta} \in \langle p_{\eta, \ell}: \ell < \omega \rangle$ a maximal antichain of $\mathbb{Q}.$ 
        \end{enumerate}
    \end{enumerate}
    
    Clearly $\Vert \mathbf{A} \Vert \leq \aleph_{1}.$ For $\eta \in Y, \iota_{\eta, \ell}  < 2,$ we just have to guarantee: 
    
    \begin{enumerate}
        \item[$(\ast)_{2}$] for each $\alpha < \omega_{1}, \mathbf{a} \in \mathbf{A}_{\alpha}$ for some $ {\beta} \in  {[\beta_Y, 
        \omega_{1}]}$ there  is $\iota <  {2}, q \in \mathbb{Q}$ above $ {\alpha}$ and $B' \in \mathbf{A}_{ {\beta}}$ such that $B' \leq B$ and $q \Vdash$``$(\forall \eta \in Y)($ if $p_{\eta, \ell} \in \mathbf{G}$ then $\iota_{\eta, \ell} = \iota)$''.   
    \end{enumerate}
    
    Why this suffice is clear. 
    
    Why this is possible to carry as in earlier proof (using ``$\mathbb{Q}$ preserve some $P$-point'' (or less)). 
\end{PROOF}

\noindent{\bf Acknowledgement:}\quad 

We thank Alan Dow for asking me about~\ref{boxplus2b} and~\ref{boxplus2c} and for some comments and Andrzej Ros{\l}anowski for much help.



\providecommand{\WileyBibTextsc}{}
\let\textsc\WileyBibTextsc
\providecommand{\othercit}{}
\providecommand{\jr}[1]{#1}
\providecommand{\etal}{~et~al.}

\bibliographystyle{amsalpha}
\bibliography{shlhetal}

\end{document}